%% file: Non-homomorphisms.tex
\documentclass[11pt]{amsart}
\usepackage{geometry} 
\usepackage{graphicx}
\usepackage{amssymb}
\usepackage[all]{xy}
\usepackage[dvipsnames]{xcolor}
\usepackage{comment}
\usepackage{pinlabel}
\usepackage[inline]{enumitem}
\usepackage{tikz}
\usepackage{hyperref}
\usepackage[nameinlink]{cleveref}
\usepackage{cite}

\hypersetup{colorlinks=true,linkcolor=teal,citecolor=purple}

\newtheorem{theorem}{Theorem}[section]
\newtheorem{lemma}[theorem]{Lemma}
\newtheorem{proposition}[theorem]{Proposition}
\newtheorem{corollary}[theorem]{Corollary}

\newtheorem{conjecture}[theorem]{Conjecture}
\theoremstyle{definition}
\newtheorem{definition}[theorem]{Definition}
\newtheorem{remark}[theorem]{Remark}
\newtheorem{example}[theorem]{Example}

\makeatletter
\newtheorem*{rep@theorem}{\rep@title}
\newcommand{\newreptheorem}[2]{%
\newenvironment{rep#1}[1]{%
 \def\rep@title{\Cref{##1}}%
 \begin{rep@theorem} \itshape}%
 {\end{rep@theorem}}}
\makeatother

\newreptheorem{theorem}{Theorem}

\newcommand{\Z}{{\ensuremath{\mathbb{Z}}}}
\newcommand{\Q}{{\ensuremath{\mathbb{Q}}}}
\newcommand{\mfs}{\mathfrak{s}}
\newcommand{\mft}{\mathfrak{t}}

\newcommand{\dtwist}{\underline{d}}
\newcommand{\lk}{\operatorname{lk}}
\newcommand{\spinc}{\text{Spin}^c}

\author[Tye Lidman]{Tye Lidman}
\thanks{The first author was partially supported by a Sloan Fellowship and NSF grants DMS-1709702 and DMS-2105469.}
\address{Department of Mathematics, North Carolina State University, Raleigh, NC 27607}
\email{tlid@math.ncsu.edu}

\author[Allison N. Miller]{Allison N.  Miller}
\thanks{The second author was partially supported by NSF grant DMS-1902880.}
\address{Department of Mathematics and Statistics, Swarthmore College, Swarthmore, PA 19081}
\email{amille11@swarthmore.edu}

\author[Juanita Pinz\'on-Caicedo]{Juanita Pinz\'on-Caicedo}
\thanks{The third author was partially supported by Simons Collaboration grant 712377.}
\address {Department of Mathematics, University of Notre Dame, Notre Dame, IN 46556}
\email{jpinzonc@nd.edu}

\title{Linking number obstructions to satellite homomorphisms}

\numberwithin{equation}{section}

\begin{document}

\begin{abstract}
We prove that satellite operations that satisfy a certain positivity condition and have winding number other than one  are not homomorphisms.  The argument uses the $d$-invariants of branched covers.  In the process, we prove a technical result relating $d$-invariants and the Torelli group which may be of independent interest.  
\end{abstract}

\maketitle

\section{Introduction}\label{sec:intro}
\input{introductionprime2}

\section*{Organization}
In \Cref{sec:topology} we describe certain cyclic branched covers of satellite knots in terms of (3-manifold, framed link) pairs and define and prove results about null-homologous surgery, twisting along an unknot, and clasper surgery.  In ~\Cref{sec:d-invariants} we constrain how each of these operations can change the $d$-invariants of the associated 3-manifolds and prove \Cref{thm:d-torelli}.  Finally,  in ~\Cref{sec:proof-main} we prove our main results,~\Cref{thm:mainnullhom} and~\Cref{thm:maincomp}.\\

\section*{Acknowledgements} 
We would like to thank Chuck Livingston for helpful comments on a previous draft. 

\section{Topological constructions}\label{sec:topology}
\input{topology}


\section{$d$-invariants}\label{sec:d-invariants}
\input{d-invariants}

\section{Proofs of the main theorems}\label{sec:proof-main}
\input{proof-of-theorem}

\bibliography{references}
\bibliographystyle{plain}

\end{document}

%% file: introductionprime2.tex
Two oriented knots in $S^3$ are called concordant if they cobound a smoothly embedded annulus in $S^3 \times I$.  The relation of concordance is an equivalence relation on the set of knots,  and the resulting collection of equivalence classes is denoted by $\mathcal{C}$.  There have been many applications of concordance to low-dimensional topology: for example,  the existence of a knot that is non-trivial in $\mathcal{C}$ but bounds a locally flat topologically embedded disk in $B^4$ implies the existence of an exotic $\mathbb{R}^4$ \cite{GS}.  
Interestingly, $\mathcal{C}$ takes on the structure of an abelian group under connected sum~\cite{FoxQuickTrip} and there has been a large amount of work on analyzing the algebraic structure of this group~\cite{FoxMilnor, Murasugi, LevineInvariants,LevineKnotCobordismGroups, LivingstonNaikObstructing,CassonGordonOnSlice,CassonGordonCobordism,  JiangBoJu,  LivingstonOrder2}. The classical satellite construction gives well-defined functions on $\mathcal{C}$ which are much studied \cite{CochranDavisRay, CochranHarveyGeometry,  CochranHarveyLeidy, DavisRay, HeddenPinzon, LevineNonsurjective,  MillerPiccirillo,  RaySatelliteIterates}. However,  the following conjecture of Hedden suggests that satellite maps on $\mathcal{C}$ almost never interact nicely with the connected sum group structure.  

\begin{conjecture}[Hedden \cite{BIRS16},\cite{MPIM16}]\label{conj:hedden}
Let $P$ be a pattern knot in the solid torus.  Let $[P]: \mathcal{C} \to \mathcal{C}$ be the associated operation which takes $[K]$ to $[P(K)]$.  If $[P]$ is a group homomorphism, then $[P]$ is either the zero, the identity,  or the reversal map. 
\end{conjecture}

Many patterns can be obstructed from inducing a homomorphism by the unexciting requirement that $P(U)$ must be a slice knot.  However,   if $P(U)$ is slice then classical invariants such as the Alexander polynomial and the Tristram-Levine signature function cannot obstruct $P$ from inducing a homomorphism on $\mathcal{C}$~\cite{Litherland, Allison}. 
The first interesting non-homomorphism in the literature was therefore not identified until work of Gompf~\cite{Gompf},  who showed that the Whitehead pattern does not induce a homomorphism on $\mathcal{C}$,  despite being identically zero on the topological concordance group.  
Later,  Levine~\cite{LevineNonsurjective} and Hedden~\cite{HeddenCablingII} used the $\tau$ invariant from Heegaard Floer homology to show that the Mazur pattern and $(n,1)$ cables ($n>1$) also do not induce homomorphisms. 
 Most recently, A. Miller \cite{Allison} used Casson-Gordon signatures to give an obstruction to patterns inducing a homomorphism even on the topological concordance group.   

Our first main result already implies Hedden's conjecture for a large collection of patterns.  For a pattern $P$ in the solid torus and prime power $q$,  the $q$-fold cyclic branched cover of $S^3$ along $P(U)$ is a rational homology sphere $\Sigma_q(P(U))$. When $q$ divides the winding number of $P$,  the curve $\eta$ that bounds a meridional disk in the solid torus lifts to a $q$-component framed oriented link $\eta_1 \cup \dots \cup \eta_q$ in $\Sigma_q(P(U))$.  For any two distinct components $\eta_i$ and $\eta_j$ the linking number $\lk(\eta_i, \eta_j)$ is a rational number.   

\begin{theorem}\label{thm:mainnullhom}
Let $P$ be a pattern in the solid torus with winding number $w$,  and let $q$ be a prime power dividing $w$. 
Suppose that the lifts $\eta_1, \dots,\eta_q$ of $\eta$ to $\Sigma_q(P(U))$
are  null-homologous in $\Sigma_q(P(U))$.  
If   $\lk(\eta_i,\eta_j) \geq 0$ for all $i \neq j$, but is not identically zero,  then $P$ does not induce a homomorphism of the smooth concordance group.
\end{theorem}
Since a pattern $P$ induces a homomorphism exactly when the mirrored pattern $\bar{P}$ does, ~\Cref{thm:mainnullhom} can also be used to obstruct a pattern with $\lk(\eta_i, \eta_j) \leq 0$ for all $i \neq j$ from inducing a homomorphism. 

\begin{example}
Let $P$ denote the positive Whitehead doubling pattern.  Then, $\Sigma_2(P(U))$ is $S^3$ and $\eta_1 \cup \eta_2$ is the $-T(2,4)$ torus link.  Similarly,  for any $n>1$ and $q$ any prime power dividing $n$,  we have that $\Sigma_q(C_{n,1}(U))=S^3$ and the lifts of $\eta$ are the $T(n,q)$ torus link. 
We therefore see that \Cref{thm:mainnullhom} recovers the results of~\cite{Gompf} and~\cite{HeddenCablingII}  that Whitehead doubling and the $(n,1)$ cable map do not induce homomorphisms on concordance. 
\end{example}

Previous work has shown that if $\eta_1, \dots, \eta_q$ generate the first homology of $\Sigma_q(P(U))$,  then $P$ does not induce a homomorphism ~\cite[Theorem A]{Allison}.  
Our work was motivated in part by a desire to obstruct homomorphisms when each component $\eta_1, \dots, \eta_q$ is null-homologous,  but we are also able to  strengthen ~\Cref{thm:mainnullhom} to the case when the lifts $\eta_1, \dots,\eta_q$ are not null-homologous in $\Sigma_q(P(U))$. 
\begin{theorem}\label{thm:mainextended}
Let $P$ be a pattern in the solid torus with winding number $w$,  and let $q$ be a prime power dividing $w$.  Let $\eta_1, \dots,\eta_q$ denote the lifts of $\eta$ to $\Sigma_q(P(U))$,  and let $n$ denote the order of $[\eta_1]$ in $H_1(\Sigma_q(P(U)))$.
Suppose that either $n$ is odd or $w$ is a nonzero multiple of $n$.
If  $\lk(\eta_i,\eta_j) \geq 0$ for all $i \neq j$, but is not identically zero,
then $P$ does not induce a homomorphism of the smooth concordance group.
\end{theorem}

\Cref{thm:mainextended} arises as an almost immediate corollary of the following somewhat peculiar result.  Note that we say a pattern $P$ induces a \emph{pseudo-homomorphism} if $P(U)$ is slice and  $P(-K)$ is concordant  to $-P(K)$ for all knots $K$.  There exist pseudo-homomorphisms of all odd winding numbers, for example the so-called ``$(n,1)$ alternating cable,'' given by the $n$-strand braid $\prod_{i=1}^{n-1} \sigma_i^{(-1)^{i+1}}$~\cite{Allison}, see \Cref{fig:composition}.

\begin{theorem}\label{thm:maincomp}
Let $P$ be a pattern in the solid torus with winding number $w$ and $q$ a prime power dividing $w$. Denote the $q$ lifts of $\eta$ to $\Sigma_q(P(U))$ by $\eta_1,\dots,\eta_q$. 
Suppose that for $i \neq j$, the numbers $\lk(\eta_i, \eta_j)$ are nonnegative and not identically zero,  and let $n$ be the order of $[\eta_1]$ in $H_1(\Sigma_q(P(U))$. 
If $R$ is any pattern whose winding number is a nonzero multiple of $n$, then the composition $P \circ R$ does not induce a pseudo-homomorphism. 
\end{theorem}

\begin{proof}[Proof of ~\Cref{thm:mainextended},  assuming~\Cref{thm:maincomp}.]
If $n$ is odd,  let $Q$ be the ``$(n,1)$ alternating cable'' given by the $n$-strand braid $\prod_{i=1}^{n-1} \sigma_i^{(-1)^{i+1}}$.  Note that $Q(U)=U$  is slice and $Q$ is isotopic to $-Q$ in the solid torus,  and hence $Q(-K)$ is isotopic to $-Q(K)$ for all knots $K$~\cite{Allison}.  So $Q$ does induce a pseudo-homomorphism, but 
\Cref{thm:maincomp}  implies that $P \circ Q$ does not induce a pseudo-homomorphism. We  therefore have our desired conclusion. 

If $w$ is a nonzero multiple of $n$, then \Cref{thm:maincomp} implies that $P \circ P$ does not induce a pseudo-homomorphism,  and so we obtain our desired conclusion.
\end{proof}

\Cref{thm:mainextended} gives us a particularly nice consequence when $q=2$. 
\begin{corollary}
Let $P$ be a pattern in the solid torus with even winding number such that the two lifts of $\eta$ to $\Sigma_2(P(U))$ have nonzero rational linking number. 
Then $P$ does not induce a homomorphism. 
\end{corollary}
\begin{proof}
The order of the first homology of the double branched cover of $S^3$ along a knot is always odd, and hence the order of $[\eta_1]$ in $H_1(\Sigma_2(P(U)))$ is odd. 
\end{proof}

The rough strategy to prove \Cref{thm:mainnullhom} is as follows.  
If $P$ acts as a homomorphism, then the knot $J_k=P(T_{2,2k+1}) \# P(-T_{2,2k+1})$ must be slice, and so the $q$-fold cyclic branched cover of $J_k$ would bound a rational homology ball, and hence have some Heegaard Floer $d$-invariants vanishing. (For the relevant background on the $d$-invariants see \Cref{sec:d-invariants}).  In particular,  this would imply that the maximal $d$-invariant of $\Sigma_q(J_k)$ would be non-negative. The linking hypothesis allows us to show that for sufficiently large $k$,  all  $d$-invariants of $\Sigma_q(J_k)$ are highly negative,   thereby obstructing $P$ from inducing a homomorphism. 

In order to carry out the $d$-invariant bound, we prove the following technical result which may be of interest to those working in Heegaard Floer homology and/or mapping class groups.  
\begin{theorem}\label{thm:d-torelli}
Let $\phi$ be an element of the Torelli group for a closed surface $\Sigma_g$.  Express $\phi$ as a product of $N$ separating Dehn twists and/or bounding pair maps.  There exists a constant $C_N$ depending only on the length of this word (and $g$) 
with the following property.  Let $Y$ be any homology sphere and $Y_\phi$ the result of removing an embedded (parameterized) genus $g$ surface in $Y$ and regluing by $\phi$.  Then, 
\[
|d(Y) - d(Y_\phi)| \leq C_N.
\]
\end{theorem}

There is a similar statement for performing such ``Torelli surgeries'' in rational homology spheres, but we do not state it here for ease of reading.
In forthcoming work of Afton, Kuzbary, and Lidman, this relationship with the Torelli group will be explored further.

\subsection{Proof strategy for ~\Cref{thm:mainnullhom}}\label{sec:proofoutline}

The (3-manifold, framed link) pair $(\Sigma_q(P(U)), \cup_{i=1}^q \eta_i)$ contains all the information necessary to construct $\Sigma_q(P(K))$ for any knot $K$:
\begin{align*}
\Sigma_q(P(K))= \left(\Sigma_q(P(U)) \smallsetminus \bigcup_{i=1}^q \nu(\eta_i) \right) \cup \bigcup_{i=1}^q E_K=: (\Sigma_q(P(U)), \bigcup_{i=1}^q \eta_i) *E_K
\end{align*}
where $E_K$ denotes the exterior of $K$ and $\partial E_K$ is glued to $\partial\overline{\nu(\eta_i)}$ by identifying $\lambda_K$ with $\mu_{\eta_i}$ and $\mu_K$ with the longitude of $\eta_i$ obtained by lifting a 0-framed longitude of $\eta$.  (See ~\Cref{sec:topology} for more details.) 

Our key idea is to construct a cobordism from $(\Sigma_q(P(U)), \bigcup_{i=1}^q \eta_i)$ to $(S^3, H_q)$,  where $H_q$ is the split union of the $(-1)$-framed positive Hopf link with a $(q-2)$-component 0-framed unlink.  We use three types of modifications,  discussed in~\Cref{sec:topology}: null-homologous surgeries,  simple negative definite cobordisms, and clasper surgeries.  In~\Cref{sec:d-invariants},  we analyze the ways these three operations relate the $d$-invariants of $\Sigma_q(P(K))=(\Sigma_q(P(U)), \bigcup_{i=1}^q \eta_i) *E_ K$ to those of $\Sigma_2(C_{2,1}(K))= (S^3, H_q) *E_K$. We are able to conclude that there exists a constant $c_{P,q}$ depending only on $P$ and $q$,  such that for any knot $K$ one has \[ d_{\scriptscriptstyle\max}(\Sigma_q(P(K))) \leq d_{\max}(\Sigma_2(C_{2,1}(K)))+c_{P,q}.\]
We then use well-known formulas to observe that
\begin{align*}
\lim_{k \to \infty}  d_{\max}(\Sigma_2(C_{2,1}(T_{2,2k+1})))+ d_{\max}(\Sigma_2(C_{2,1}(-T_{2,2k+1})) =-\infty,
\end{align*} 
and use this to see that for sufficiently large $k$, 
\[d_{\max}(\Sigma_q(P(T_{2,2k+1})\#P(-T_{2,2k+1})))< 0.\]
We then conclude that $\Sigma_q(P(T_{2,2k+1})\#P(-T_{2,2k+1}))$ cannot bound a rational homology ball and hence that $P(T_{2,2k+1})\#P(-T_{2,2k+1})$ is not slice. 

\begin{remark}
This project was an offshoot of an attempt to approach the problem of concordance homomorphisms using Chern-Simons theory as applied previously by Hedden, Kirk, and Pinz\'on-Caicedo (see for example \cite{HeddenKirk, HeddenPinzon}).  The key idea would be to build a certain negative definite cobordism from $\Sigma_q(P(T_{r,s}) \# P(-T_{r,s}))$ to a 3-manifold
built out of surgeries on $T_{r,s}$.  Roughly, as the parameters of the torus knots grow, the minimal Chern-Simons invariants of this manifold would decrease, and once small enough, would preclude the branched cover from bounding a rational homology ball.  The input of Heegaard Floer homology allowed us to use this argument, with the unboundedness of $d(\Sigma_2(P(T_{2,2k+1}) \# P(-T_{2,2k+1})))$ playing the role of ``small minimal Chern-Simons invariant'', while simultaneously understanding the behavior of the invariant under other topological operations, such as Dehn surgery and clasper surgery (see~\Cref{sec:d-invariants} for concrete statements).  It would be interesting to prove some of the analogous statements in ~\Cref{sec:d-invariants} in the realm of Chern-Simons theory.
\end{remark}

%% file: topology.tex
Since we base our argument in a specific topological relationship involving cyclic branched covers of satellite knots, we begin this section with a convenient decomposition of these covers, originally due to Seifert~\cite{seifert-covers}. Given a pattern $P$ in a parametrized solid torus $S^1 \times D^2$, we can take the standard embedding of $S^1 \times D^2$ into $S^3$ and consider the 2-component link given by $P \cup (* \times \partial D^2)$. 
This process is reversible: given an ordered link $P \cup \eta$ in $S^3$ such that $\eta$ is unknotted, we can canonically identify $S^3 \smallsetminus \nu(\eta)$ with a parametrized solid torus $S^1 \times D^2$ by requiring that a meridian of $\eta$ be identified with $S^1 \times \{*\}$ for some $* \in \partial D^2$.

One virtue of this link description of a pattern is that it leads to a useful decomposition of the exterior of a satellite knot: 
\begin{align}\label{eqn:extofpk} E_{P(K)}= \left(S^3 \smallsetminus \nu(\eta \cup P)\right) \cup E_K \end{align}
where  $\partial E_K$ is identified with $\partial\overline{\nu(\eta)}$ in such a way that $\lambda_K$, the 0-framed longitude of $K$, is identified with $\mu_{\eta}$ and similarly $\mu_K$ is identified with $\lambda_\eta$.

This induces a convenient decomposition of certain branched covers of $S^3$ along $P(K)$, as follows. 
Letting $q$ be a positive divisor of the winding number $w= \lk(P, \eta)$, the curve $\eta$ lifts to $q$ distinct curves $\eta_1, \dots, \eta_q$ in the $q$-fold cyclic branched cover $\Sigma_q(P(U))$ and the $q$-fold cover of $S^3$ branched over $P(K)$ decomposes as  
\begin{align}\label{eqn:cyclicofpk}
	\Sigma_q(P(K))=
	\left(\Sigma_q(P(U))\setminus \nu(\eta_1\cup \ldots \cup\eta_q)\right)
	\cup \bigcup_{i=1}^q E_K^{(i)}. 
\end{align}
Note that  $E_{K}^{(i)}$ is attached along $\partial\overline{\nu(\eta_i)}$ via a gluing map that identifies $\mu_K^{(i)}$ with the $i$-th lift of a 0-framing of $\eta$, and $\lambda_K^{(i)}$ with the $i$-th lift of a meridian of $\eta$.

Motivated by this topological description of covers branched along satellite knots, we introduce the following notation: given a framed $q$-component link $L$ in a 3-manifold $Y$ and a knot $K$ in $S^3$,  we define 
\begin{align}~\label{eqn:starnotation}
	(Y,L)*E_K=(Y \smallsetminus \nu(L)) \cup \bigcup_{i=1}^q E_K^{(i)}
\end{align}
where $E_K$ denotes the exterior of $K$ and $\partial E_K^{(i)}$ is identified with $\partial \nu(L_i)$ such that the Seifert longitude $\lambda_K$ is identified with $\mu_{L_i}$ and that $\mu_K$ is identified with the framing curve $\gamma_{L_i}$.\footnote{For example,  if $L_i$ is null-homologous and $f_i$ is the framing number of $L_i$, then $\gamma_{L_i}$ is the push-off of $L_i$ to $\partial N(L_i)$ such that $\lk (L_i,\gamma_{L_i})=f_i$.} Moreover, the decomposition from \Cref{eqn:cyclicofpk} can thus be denoted by $$\left(\Sigma_q(P(U)),\eta_1\cup \ldots \cup\eta_q \right)*E_K.$$  

We note for later reference that if $Y$ is an integer homology sphere and $L$ is any link in $L$,  the manifold $(Y,L)*E_K$ must also be an integer homology sphere,  since knot exteriors are homology solid tori. 

\begin{example}\label{exl:simplestar}
	Let $U$ be an $n$-framed unknot in $S^3$.  Then for any knot $K$ we have that \[ (S^3,U)*E_K= E_U \cup E_K,\]
	where $\lambda_K$ is identified with $\mu_U$ and $\mu_K$ is identified with a curve representing $n [\mu_{U}]+ [\lambda_U]$ in $H_1(\partial E_U)$.  Therefore the curve representing $[\mu_K]-n[\lambda_K]$ in $H_1(E_K)$ is identified with $\lambda_U$,  which bounds a meridional disc in $E_U=S^1 \times D^2$,  and we have that  $$(S^3,U)*E_K=S^3_{-1/n}(K).$$  
\end{example}

\begin{example}\label{exl:dbc-c21}
	Let $K$ be any knot.  The Akbulut-Kirby algorithm \cite{AkbulutKirby} applied to the M\"obius band bounded by $C_{2,1}(K)$ gives that $$\Sigma_2(C_{2,1}(K))= S^3_{+1}(K\#K^r).$$
	Alternately,  we obtain that $$\Sigma_2(C_{2,1}(K))= (S^3, H)* E_K,$$  where $H$ denotes the $(-1)$-framed positive Hopf link,  as follows. 
	
	Recall that the exterior of $C_{2,1}(K)$ is obtained from the exterior of $C_{2,1} \cup \eta$ by gluing in $E_K$ so that $\lambda_K$ is identified with $\mu_\eta$ and $\mu_K$ is identified with a curve representing  $\lambda_\eta$.  In~\Cref{fig::cable},  we see two diagrams for $C_{2,1} \cup \eta$: the usual one on the left and a diagram where $C_{2,1}$ is the standard unknot on the right.  
	\begin{figure}[h]
		\centering
		\def\svgwidth{0.2\textwidth}
		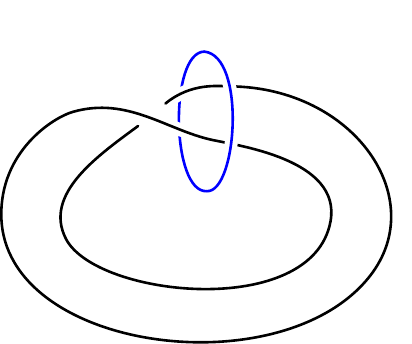\raisebox{1cm}{$\cong$}
		\def\svgwidth{0.2\textwidth}
		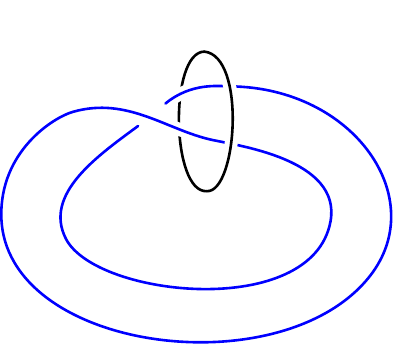\hspace{1cm}
		\caption{The $(2,1)$-cable pattern. }\label{fig::cable}
	\end{figure}
	In~\Cref{fig::cablecover},  we branch over the disc bounded by $C_{2,1}$ to obtain a schematic diagram for $\Sigma_2(C_{2,1}(K)$. 
	\begin{figure}[h]
		\centering
		\def\svgwidth{0.2\textwidth}
		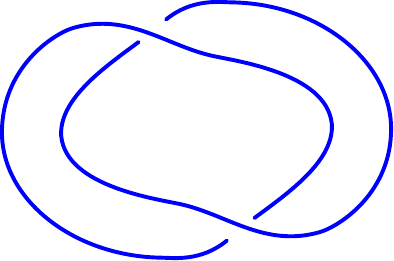
		\caption{ A diagram for the 2-fold branched cover of $C_{2,1}(K)$,  where the $-1$ labels indicate that the copies of $E_K$ are being glued in to identify $\mu_K$ with the $(-1,1)$ curve of each $\eta_i$. }\label{fig::cablecover}
	\end{figure} Note that the $-1$ labels on the two lifts of $\eta$ come from the fact that the curve $\eta$ on the right of~\Cref{fig::cable} has writhe equal to $+1$,  and are not surgery framings, but rather indicate that $\mu_K$ must be glued to the $(-1,1)$-curve of each $\eta_i$.  We therefore obtain as desired that $\Sigma_2(C_{2,1}(K))= (S^3, H)* E_K$.
\end{example}

We now remind the reader that for disjoint oriented curves $\gamma_1$ and $\gamma_2$ in a rational homology sphere $Y$,  there is an associated linking number $\lk_Y(\gamma_1, \gamma_2) \in \mathbb{Q}$~\cite[Chapter 10,  Section 77]{st} that depends on the isotopy classes of the two curves.\footnote{The reader may be more familiar with the $\Q/\Z$ torsion linking form: the image of $\lk_Y(\gamma_1, \gamma_2)$ in $\mathbb{Q}/ \mathbb{Z}$ is an invariant of the homology classes $[\gamma_1], [\gamma_2] \in H_1(Y; \mathbb{Z})$,  but with the additional information of the isotopy class the rational number is well-defined.} We remark that when one of $\gamma_1$ or $\gamma_2$ is null-homologous, as is guaranteed for example when $Y$ is an integer homology sphere, then $\lk_Y(\gamma_1, \gamma_2)$ is an integer.  When it is clear from context what the ambient 3-manifold is, we shall sometimes abbreviate $\lk_Y(\gamma_1, \gamma_2)$ as simply $\lk(\gamma_1,\gamma_2)$. 

Now,  returning to our consideration of $\eta_1\cup \ldots \cup\eta_q$ in $\Sigma_q(P(U))$,  we observe that orienting one of the $\eta_i$ determines an orientation for all the rest  by requiring that the link be covering transformation invariant as an oriented link. Therefore,   for $i \neq j$ we have that $\lk_{\Sigma_q(P(U))}(\eta_i,  \eta_j) \in \mathbb{Z}$ is well-defined. 

Until~\Cref{sec:nonnullhomologous}, we assume that each (or, equivalently, some) $\eta_i$ is null-homologous in $\Sigma_q(P(U))$. It follows that the longitudes of each $\eta_i$ canonically correspond to elements of $\mathbb{Z}$, by requiring that the Seifert longitude corresponds to $0$.  Define $\text{fr}(\eta_i)$ to be the integer corresponding to the $i$th lift of the 0-framed longitude of $\eta$.\\

\begin{lemma}~\label{lem:linkingframing}
	Given the notation above, 
	$ \text{fr}(\eta_i)= - \sum_{j \neq i} \lk(\eta_i, \eta_j).$ 
\end{lemma}
\begin{proof}
	Denote the $i$-th lift of the 0-framed longitude of $\eta$ by $\lambda_i$.
	Recall that $\eta$ is an unknot in $S^3$,  and hence bounds a disc $D$ that we may assume is transverse to $P$.  In particular, the 0-framed longitude of $\eta$ is given by pushing $\eta$ into the interior of $D$. 
	The lift of $D$ to $\Sigma_q(P(U))$ is a (possibly disconnected) surface $F$ with boundary $\cup_{k=1}^q \eta_k$,  and $\lambda_i$ is given by pushing $\eta_i$ into the interior of $F$.  Modifying $F$ by removing the boundary annulus near $\eta_i$ gives a surface $\hat{F}$ with boundary $\partial \hat{F}=\lambda_i \cup \bigcup_{j \neq i} \eta_j$.  It follows that $\lambda_i$ is homologous to $-\cup_{ j \neq i} \eta_j$ in the complement of $\eta_i$,  and therefore that 
	\[
	\text{fr}(\eta_i)= \text{lk}(\eta_i,  \lambda_i)= \text{lk}(\eta_i, -\cup_{j \neq i} \eta_j)
	=-\sum_{j \neq i} \text{lk}(\eta_i, \eta_j),\]
	as desired. 
\end{proof}

We will need three ways to modify a (3-manifold,  framed link) pair $(Y,L)$: null-homologous surgery,  twisting along an unknot,  and clasper surgery.  (The reader should think of $Y$ as $\Sigma_q(P(U))$ and $L$ as the lift of $\eta$.)   
We now define these operations,  analyze how they change the linking-framing information of $L$,  and prove some key results for later use.  
We remark that in each case an operation from $(Y,L)$ to $(Y', L')$ will induce a parallel operation from $(Y, L)*E_K$ to $(Y', L')*E_K$. 

\subsection{Null-homologous surgery}

 Let $L$ be a link in $Y$ and $\gamma$ a curve that is null-homologous in $Y \smallsetminus \nu(L)$.  Observe that  surgery along $\gamma$ transforms $(Y, L)$ to a new $(Y', L')$,  with the key property that $L'$
has the same  linking numbers and framings as $L$.  We will call this operation a {\em null-homologous surgery} on a 3-manifold, link pair.

\begin{proposition}\label{prop:qhs3-to-s3}
Let $Y$ be a rational homology sphere and $L$ a framed null-homologous link in $Y$.  
There exists a framed link in $L'$ in $S^3$ with the same linking-framing information as $L$, a natural number $k$,  and a sequence of natural numbers $m_1, \dots, m_k$ with the property that
 there exists a finite sequence of null-homologous surgeries from $(S^3,L')$ to $(Y \#_{i=1}^k L(m_i,1), L)$. 
\end{proposition}

By the phrase ``a sequence of null-homologous surgeries'', we mean that there exists a link $J = J_1 \cup \ldots \cup J_r$ such that $J$ is nullhomologous in $Y \smallsetminus \nu(L)$, and for each $i$, $J_{i+1}$ is nullhomologous in the complement of the image of $L$ after surgery on $J_1, \ldots, J_i$.  We now consider the following equivalent formulation.  
\begin{lemma}~\label{lem:qhs3toanyzhs3}
Let $L = L_1 \cup \ldots \cup L_n$ be a null-homologous framed link in a rational homology sphere $Y$.  Then there exist two disjoint framed links $L' = L'_1 \cup \ldots \cup L'_n$, $J = J_1 \cup \ldots \cup J_k$ in $S^3$ with $\lk_{Y}(L_i,L_j) = \lk_{S^3}(L'_i, L'_j)$ for all $i \neq j$ which satisfies the following property.  
\begin{enumerate}
\item $\lk_{S^3} (J_i, J_j) = \lk_{S^3}(J_i,L'_j) = 0$ for all $i, j$ (i.e. $J$ has pairwise linking zero and is null-homologous in the exterior of $L'$);
\item framed surgery on $J$ produces the connected sum of $Y$ and lens spaces of the form $L(m,1)$;
\item the image of $L'$ under surgery on $J$ is contained in the $Y$-summand and agrees with $L$ as a framed link. 
\end{enumerate}
\end{lemma}
\Cref{lem:qhs3toanyzhs3} immediately establishes \Cref{prop:qhs3-to-s3}: since $(Y,L)$ is obtained from $(S^3, L')$ by surgery on a link with pairwise linking zero and trivial linking with $L'$, we see that after applying surgery on each component of $J$ the sublink of the remaining components of $J$ will still be null-homologous. 

\begin{proof}[Proof of \Cref{lem:qhs3toanyzhs3}]
After connected sum with lens spaces of the form $L(m,1)$, $Y$ can be expressed as surgery on a framed link in $S^3$ with pairwise linking zero \cite[Lemma 2.4]{Lidman}.  This link will be $J$.  Viewing $Y \#_i L(m_i,1)$ as surgery on $J$, we can view the link $L$ as a link $L''$ in $S^3$ which is disjoint from $J$.  Note that $L''$ may have linking with $J$.  

We first claim that we may handle-slide the components of $L''$ over $J$ so that they do not link $J$ algebraically.  To see this, consider the linking matrix $A_J$ associated to the framed link $J$.  Now, let $\vec{x}^{(1)}$ be the integral vector with $x^{(1)}_i = \lk_{S^3}(L''_1, J_i)$.  Note that if we handle-slide $L''_1$ over $J$ (where the handle-slide happens away from the other components of $L''$), we do not change the isotopy type of $L''$ as a link in $Y \#_i L(m_i,1)$, but we do change the corresponding vector $\vec{x}^{(1)}$.  In particular, when handle-sliding over the $i$th component of $J$, we add the $i$th column of $A_J$ to $\vec{x}^{(1)}$.  Because $L$ is null-homologous in $Y$, and hence $Y \#_i L(m_i,1)$, this means that the vector $\vec{x}^{(1)}$ is contained in the integral column space of $A_J$.  Therefore, there exists a sequence of handle-slides of $L''_1$ over $J$ to produce a new link with $\lk_{S^3}(L''_1,J_i) = 0$ for all $i$.  We now repeat the same argument with $L''_2$.  Since $\lk_{S^3}(L''_1, J_i) = 0$ for all $i$, handle-sliding $L''_2$ over $J$ does not change the linking of $L''_1$ with $J$.  Continue this process for all components of $L''$.  The result after the sequence of handle-slides is the claimed link $L'$.  

It remains to show that $\lk_{Y}(L_i, L_j) = \lk_{S^3}(L'_i, L'_j)$ for all $i \neq j$.  Since $L'$ has no linking with $J$, we can find a Seifert surface $S_i$ for each component of $L'_i$ which is disjoint from $J$; we know the algebraic intersection number of each component of $J$ with $S_i$ is zero, and we can tube the $S_i$ to make sure that the geometric intersection is zero as well.  (Note we will have some $S_i$'s intersecting $L'_j$'s with $i \neq j$.)  Since $S_i$ is disjoint from $J$, we see that $\lk_{S^3}(L'_i, L'_j)$ is unchanged by surgery on $J$.  Consequently, $\lk_Y(L_i, L_j) = \lk_{S^3}(L'_i, L'_j)$. \end{proof}

\begin{remark}
The addition of lens spaces is not necessary if $Y$ is a homology sphere.  An example of a rational homology sphere that cannot be obtained by integral surgery on a link in $S^3$ with pairwise linking zero is the 3-manifold obtained by 0-surgery on the 2-component link $T_{2,4}$. Nevertheless, the connected sum of this rational homology sphere with $-L(2,1)=S^3_2(U)$ can be represented as surgery along a framed link with diagonal linking matrix with diagonal entries $(-2,2,2)$.\\\end{remark}

\subsection{A simple negative-definite cobordism from twisting along an unknot}~\label{sec:twistingonanunknot}
As above, let $L = L_1 \cup \ldots \cup L_n$ be a link in a 3-manifold $Y$.  Choose $i \neq j$.  Let $\alpha$ be an unknot  in $Y$ bounding a disc $D$ such that the geometric intersection of $D$ with $L_k$ is a single positively oriented point if $k=i$,  a single negatively oriented point if $k=j$,  and the empty set otherwise, as illustrated in \Cref{fig:twisting}. Then doing $-1$ framed surgery along $\alpha$  does not change $Y$,  but transforms $L$ to a link $L'$ with
\begin{equation}\label{eq:neg-def-linking}
 \lk(L_k', L_\ell')= \begin{cases}
\lk(L_k ,L_\ell) -1 & \text{if }   \{k, \ell\} = \{i, j \}
    \\  \lk(L_k, L_\ell) & \text{else}
 \end{cases}
\end{equation}
and
\begin{equation}\label{eq:neg-def-framing}
\text{fr}(L_k')= \begin{cases}
\text{fr}(L_k)+ 1 & \text{if }   k\in  \{i, j \}
    \\  \text{fr}(L_k) & \text{else.}
 \end{cases}
\end{equation}

\begin{figure}[h]
	\centering
	\def\svgwidth{0.25\textwidth}
	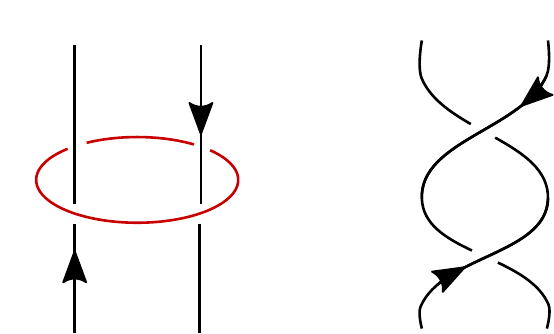
	\caption{Positive twisting along an unknot $\alpha$ changes $\lk(L_i, L_j)$ by $-1$ and $\text{fr}(L_i)$ and $\text{fr}(L_j)$ by $+1$.}
	~\label{fig:twisting}
\end{figure}

We call the 4-manifold 
\[ W= Y \times I \cup_{\alpha \times \{1\}} \text{2-handle}\] an \textit{elementary negative-definite cobordism} from $(Y,L)$ to $(Y,  L')$. Note that if $W$ is an elementary negative definite cobordism from $(Y,L)$ to $(Y', L')$ then for any knot in $S^3$ we can glue a collection of  copies of $E_K \times I$ to obtain a negative definite cobordism from $(Y,L)*E_K$ to $(Y,L')*E_K$.

\begin{proposition}\label{prop:reduce-linking}
Let $L$ be a framed $q$-component link in $S^3$ such that
\begin{enumerate}
\item $\lk(L_i, L_j)\geq 0$ for all $i \neq j$, 
\item   $\lk(L_{i_0}, L_{j_0}) \geq 1$ for some $i_0 \neq j_0$,  and
\item $fr(L_i)=- \sum_{j \neq i} \lk(L_i, L_j)$ for all $i$. 
\end{enumerate}
Then there exists a sequence of elementary negative definite cobordisms from $(S^3, L) $ to $(S^3, L')$, where  $L'$ is a framed link in $S^3$ satisfying
\[\lk(L'_i,  L'_j)= \begin{cases} 1 & \text{if } \{i, j \}=\{i_0, j_0\} \\ 0 & \text{else}. \end{cases} \quad\text{ and }\quad fr(L_i')= \begin{cases} -1 & \text{if } i \in \{i_0, j_0\} \\ \phantom{-}0 & \text{ else}. \end{cases}\] 
In particular,  for any knot $K$ there is a negative definite cobordism from $(S^3, L)*E_K$ to $(S^3, L')*E_K$. 
\end{proposition}

\begin{proof}[Proof of \Cref{prop:reduce-linking}]
By renumbering the components of $L$, we may assume for simplicity in notation that $i_0=1$ and $j_0=2$. 
We construct a $(\sum_{i \neq j} \lk(L_i, L_j)-1)$-component link $J$ in $S^3 \smallsetminus L$ as follows. 
For each $i <j$,  let $\alpha_{i,j}$ be an arc in $S^3$ connecting $L_i$ to $L_j$ and otherwise disjoint from $L$.  By transversality, we can assume that all such arcs are disjoint.  We can choose a disc $D_{i,j}$ living arbitrarily close to $\alpha_{i,j}$,  with a boundary that is homotopic to $\alpha_{i,j} \mu_{L_i} \bar{\alpha}_{i,j} \bar{\mu}_{L_j}$ in $S^3 \smallsetminus L$, as illustrated on the left and center of \Cref{fig:buildingunlink}. In particular, 
\[\lk(\partial D_{i,j}, L_k)= \begin{cases} \phantom{-}1 & k=i \\ -1 & k=j \\ \phantom{-}0 & \text{ else} \end{cases}.\]

\begin{figure}[h!]
\includegraphics[height=3cm]{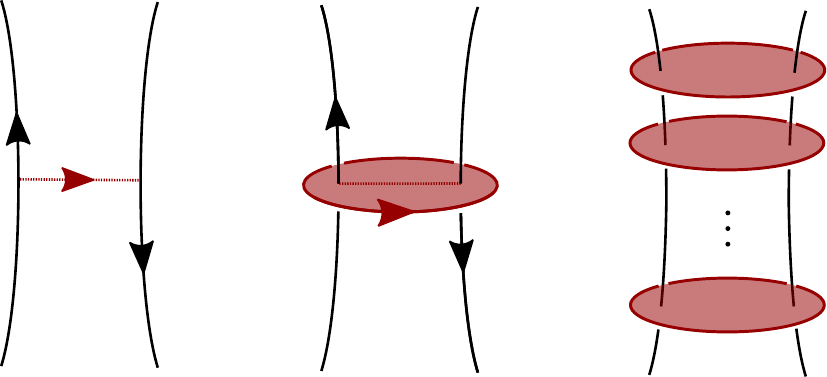}
\caption{An arc $\alpha_{i,j}$ connecting $L_i$ to $L_j$ (left);  a disc $D_{i,j}$ lying very close to $\alpha_{i,j}$ (center); and $\lk(L_i,L_j)$ parallel push-offs of $D_{i,j}$ (right).  }\label{fig:buildingunlink}
\end{figure} 

For $i<j$ such that $(i,j) \neq (1,2)$ and $\lk(L_i, L_j) \neq 0$,  let $\mathcal{D}_{i,j}$ be the disjoint union of $\lk(L_i, L_j)$ parallel push-offs of $D_{i,j}$,  as illustrated on the right of \Cref{fig:buildingunlink}. 
Let $\mathcal{D}_{1,2}$ be the disjoint union of $\lk(L_1, L_2)-1$ parallel push-offs of $D_{1,2}$ (or the empty set, if $\lk(L_1,L_2)=1$. ) 
Let $\mathcal{D}$ be the union of all of these discs,  and denote by $\mathcal{L}$ its boundary.  By construction,  $\mathcal{L}$ is an unlink in $S^3$ that is disjoint from $L$. 
By attaching $(-1)$-framed 2-handles to $S^3 \times [0,1]$ along $\mathcal{L} \subset S^3 \times \{1\}$,  we obtain a negative definite cobordism from $(S^3, L)$ to $(S^3, L')$ for some link $L'$.  It remains to show that the linking-framing information of $L'$ is as desired.  But this follows immediately from \eqref{eq:neg-def-linking},  given our choice of the number of components of $\mathcal{L}_{i,j}= \partial D_{i,j}$.
\end{proof}

\subsection{Clasper surgeries}
Although null-homologous surgeries do not change linking information, it is not true that any two links with the same pairwise linking can be transformed into each other by a sequence of null-homologous surgeries.  A simple example of this is the 3-component unlink and the Borromean rings. Nevertheless, the latter can be transformed into the former via a well-studied generalization of Dehn surgery called clasper or Borromean surgery, an operation commonly studied in quantum topology. This operation will allow us to transform the lift of $\eta$ into a simpler and familiar link, and this will in turn allow us to prove our main theorem.

\begin{definition}
A {\em clasper surgery on a 3-manifold} is the result of removing from a 3-manifold the genus three handlebody on the left of \Cref{fig:clasper} and replacing it with the handlebody on the right of the same figure,  or the reverse. 
 A {\em clasper surgery on a link} in a 3-manifold is defined analogously where the genus three handlebody must be disjoint from the link.     
\end{definition}

Clasper surgery can be equivalently defined as removing the handlebody $H$, and regluing the same $H$ with gluing map a suitable element $\psi$ of the mapping class group of the genus three surface $\partial H$. Since the blue curves on the left and right of \Cref{fig:clasper} are pairwise homologous, the map $\psi$ is an element of the Torelli group of the genus $3$ surface $S_3$. As a consequence, clasper surgery on a 3-manifold does not change the homology.

We remark that if one does clasper surgery to the unknotted (genus three) handlebody in $S^3$,  the result is still $S^3$, since the pairwise geometric linking of all curves remains the same as seen in \Cref{fig:clasper}. Similarly, modifying the standard genus three splitting for $\#_3 S^2 \times S^1$ by clasper surgery yields $\mathbb{T}^3$. Since we do not need this fact, we leave the verification as an exercise for the reader. 
\begin{figure}[h!]
\includegraphics[width=0.25\textwidth]{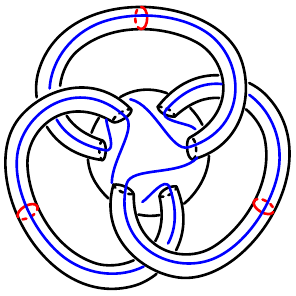}\hspace{0.75in}
\includegraphics[width=0.25\textwidth]{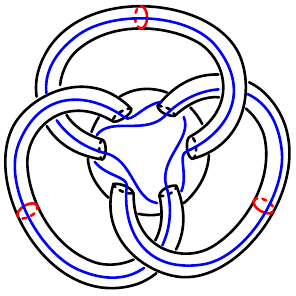}
\caption{Pictured are two Heegaard diagrams representing a clasper surgery from $S^3$ to $S^3$.  If we ignore the red curves, the handlebodies with compressing disks determined by the blue curves are exactly the genus $3$ handlebodies involved in the definition of clasper surgery. 
}
\label{fig:clasper}
\end{figure}

In the case of links, the result of clasper surgery is to tie the strands of the link ``passing through'' the one-handles of the handlebody into a Borromean link as illustrated in \Cref{fig:clasperlink}. 

\begin{figure}[h!]
\includegraphics[height=4cm]{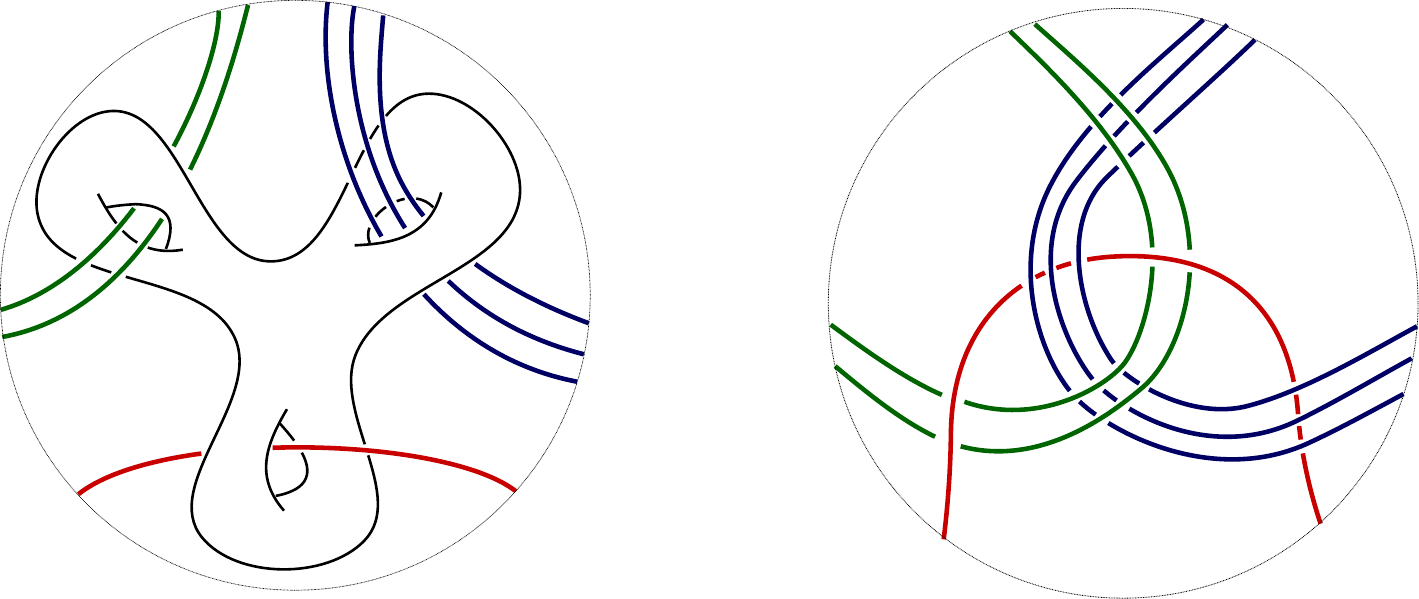}
\caption{Clasper surgery locally transforms between the links on the left and right.  Note that the depicted arcs may belong to any components of the link. }
\label{fig:clasperlink}
\end{figure}

Similar to the manifold case,  clasper surgery on a link does not change pairwise linking numbers or framings.  (However, it can change the higher Milnor linking invariants!) What is more interesting is that the converse is also true:

\begin{theorem}[Proof of Lemma 2 of \cite{Matveev}, see also \cite{MurakamiNakanishi}]\label{thm:claspers}
Let $L = L_1 \cup \ldots \cup L_n$ and $L' = L'_1 \cup \ldots \cup L'_n$ be two framed links in $S^3$ with the same number of components.  If the linking matrices for $L$ and $L'$ agree, then $L$ and $L'$ are related by a sequence of clasper surgeries.  
\end{theorem}

A special case of this is the following proposition.
\begin{proposition}\label{prop:make-standard}
Let $L$ be a framed $q$-component link in $S^3$ with the property that
\[\lk(L_i,  L_j)= \begin{cases} 1 & \text{if } \{i, j \}=\{i_0, j_0\} \\ 0 & \text{else} \end{cases}.\quad\text{ and }\quad
fr(L_i)= \begin{cases} -1 & \text{if } i \in \{i_0, j_0\} \\ \phantom{-}0 & \text{ else} \end{cases}.\]
Then there is a clasper surgery from $(S^3, L)$ to $(S^3, H_q)$, where $H_q$ is the split union of the $(-1)$-framed positive Hopf link with a $0$-framed $(q-2)$-component unlink. 
\end{proposition}

%% file: cable-pattern.pdf_tex
\begingroup%
  \makeatletter%
  \providecommand\color[2][]{%
    \errmessage{(Inkscape) Color is used for the text in Inkscape, but the package 'color.sty' is not loaded}%
    \renewcommand\color[2][]{}%
  }%
  \providecommand\transparent[1]{%
    \errmessage{(Inkscape) Transparency is used (non-zero) for the text in Inkscape, but the package 'transparent.sty' is not loaded}%
    \renewcommand\transparent[1]{}%
  }%
  \providecommand\rotatebox[2]{#2}%
  \newcommand*\fsize{\dimexpr\f@size pt\relax}%
  \newcommand*\lineheight[1]{\fontsize{\fsize}{#1\fsize}\selectfont}%
  \ifx\svgwidth\undefined%
    \setlength{\unitlength}{118.59375bp}%
    \ifx\svgscale\undefined%
      \relax%
    \else%
      \setlength{\unitlength}{\unitlength * \real{\svgscale}}%
    \fi%
  \else%
    \setlength{\unitlength}{\svgwidth}%
  \fi%
  \global\let\svgwidth\undefined%
  \global\let\svgscale\undefined%
  \makeatother%
  \begin{picture}(1,0.83442827)%
    \lineheight{1}%
    \setlength\tabcolsep{0pt}%
    \put(0,0){\includegraphics[width=\unitlength,page=1]{cable-pattern.pdf}}%
    \put(0.46353388,0.74239196){\makebox(0,0)[lt]{\lineheight{1.25}\smash{\begin{tabular}[t]{l}$\eta$\end{tabular}}}}%
    \put(0.17379593,0.27639869){\makebox(0,0)[lt]{\lineheight{1.25}\smash{\begin{tabular}[t]{l}$C_{2,1}$\end{tabular}}}}%
  \end{picture}%
\endgroup%

%% file: cable-symmetry.pdf_tex
\begingroup%
  \makeatletter%
  \providecommand\color[2][]{%
    \errmessage{(Inkscape) Color is used for the text in Inkscape, but the package 'color.sty' is not loaded}%
    \renewcommand\color[2][]{}%
  }%
  \providecommand\transparent[1]{%
    \errmessage{(Inkscape) Transparency is used (non-zero) for the text in Inkscape, but the package 'transparent.sty' is not loaded}%
    \renewcommand\transparent[1]{}%
  }%
  \providecommand\rotatebox[2]{#2}%
  \newcommand*\fsize{\dimexpr\f@size pt\relax}%
  \newcommand*\lineheight[1]{\fontsize{\fsize}{#1\fsize}\selectfont}%
  \ifx\svgwidth\undefined%
    \setlength{\unitlength}{118.59375bp}%
    \ifx\svgscale\undefined%
      \relax%
    \else%
      \setlength{\unitlength}{\unitlength * \real{\svgscale}}%
    \fi%
  \else%
    \setlength{\unitlength}{\svgwidth}%
  \fi%
  \global\let\svgwidth\undefined%
  \global\let\svgscale\undefined%
  \makeatother%
  \begin{picture}(1,0.83442827)%
    \lineheight{1}%
    \setlength\tabcolsep{0pt}%
    \put(0,0){\includegraphics[width=\unitlength,page=1]{cable-symmetry.pdf}}%
    \put(0.17075099,0.28857846){\makebox(0,0)[lt]{\lineheight{1.25}\smash{\begin{tabular}[t]{l}$\eta$\end{tabular}}}}%
    \put(0.44901186,0.75656265){\makebox(0,0)[lt]{\lineheight{1.25}\smash{\begin{tabular}[t]{l}$C_{2,1}$\end{tabular}}}}%
  \end{picture}%
\endgroup%

%% file: cable-cover.pdf_tex
\begingroup%
  \makeatletter%
  \providecommand\color[2][]{%
    \errmessage{(Inkscape) Color is used for the text in Inkscape, but the package 'color.sty' is not loaded}%
    \renewcommand\color[2][]{}%
  }%
  \providecommand\transparent[1]{%
    \errmessage{(Inkscape) Transparency is used (non-zero) for the text in Inkscape, but the package 'transparent.sty' is not loaded}%
    \renewcommand\transparent[1]{}%
  }%
  \providecommand\rotatebox[2]{#2}%
  \newcommand*\fsize{\dimexpr\f@size pt\relax}%
  \newcommand*\lineheight[1]{\fontsize{\fsize}{#1\fsize}\selectfont}%
  \ifx\svgwidth\undefined%
    \setlength{\unitlength}{113.19717407bp}%
    \ifx\svgscale\undefined%
      \relax%
    \else%
      \setlength{\unitlength}{\unitlength * \real{\svgscale}}%
    \fi%
  \else%
    \setlength{\unitlength}{\svgwidth}%
  \fi%
  \global\let\svgwidth\undefined%
  \global\let\svgscale\undefined%
  \makeatother%
  \begin{picture}(1,0.66162352)%
    \lineheight{1}%
    \setlength\tabcolsep{0pt}%
    \put(0,0){\includegraphics[width=\unitlength,page=1]{cable-cover.pdf}}%
    \put(0.8524657,0.60304676){\makebox(0,0)[lt]{\lineheight{1.25}\smash{\begin{tabular}[t]{l}$[-1]$\end{tabular}}}}%
    \put(0.22400347,-0.04038591){\makebox(0,0)[rt]{\lineheight{1.25}\smash{\begin{tabular}[t]{r}$[-1]$\end{tabular}}}}%
    \put(0.0962614,0.6257729){\makebox(0,0)[lt]{\lineheight{1.25}\smash{\begin{tabular}[t]{l}$\eta_1$\end{tabular}}}}%
    \put(0.87305911,0.02359134){\makebox(0,0)[lt]{\lineheight{1.25}\smash{\begin{tabular}[t]{l}$\eta_2$\end{tabular}}}}%
  \end{picture}%
\endgroup%

%% file: twisting.pdf_tex
\begingroup%
  \makeatletter%
  \providecommand\color[2][]{%
    \errmessage{(Inkscape) Color is used for the text in Inkscape, but the package 'color.sty' is not loaded}%
    \renewcommand\color[2][]{}%
  }%
  \providecommand\transparent[1]{%
    \errmessage{(Inkscape) Transparency is used (non-zero) for the text in Inkscape, but the package 'transparent.sty' is not loaded}%
    \renewcommand\transparent[1]{}%
  }%
  \providecommand\rotatebox[2]{#2}%
  \newcommand*\fsize{\dimexpr\f@size pt\relax}%
  \newcommand*\lineheight[1]{\fontsize{\fsize}{#1\fsize}\selectfont}%
  \ifx\svgwidth\undefined%
    \setlength{\unitlength}{159.42311085bp}%
    \ifx\svgscale\undefined%
      \relax%
    \else%
      \setlength{\unitlength}{\unitlength * \real{\svgscale}}%
    \fi%
  \else%
    \setlength{\unitlength}{\svgwidth}%
  \fi%
  \global\let\svgwidth\undefined%
  \global\let\svgscale\undefined%
  \makeatother%
  \begin{picture}(1,0.60192913)%
    \lineheight{1}%
    \setlength\tabcolsep{0pt}%
    \put(0,0){\includegraphics[width=\unitlength,page=1]{twisting.pdf}}%
    \put(-0.00474734,0.2508786){\color[rgb]{0,0,0}\makebox(0,0)[lt]{\lineheight{1.25}\smash{\begin{tabular}[t]{l}$\alpha$\end{tabular}}}}%
    \put(0.08451747,0.55365935){\color[rgb]{0,0,0}\makebox(0,0)[lt]{\lineheight{1.25}\smash{\begin{tabular}[t]{l}$L_i$\end{tabular}}}}%
    \put(0.29959582,0.55365935){\color[rgb]{0,0,0}\makebox(0,0)[lt]{\lineheight{1.25}\smash{\begin{tabular}[t]{l}$L_j$\end{tabular}}}}%
  \end{picture}%
\endgroup%

%% file: d-invariants.tex
The goal of this section is to constrain the $d$-invariants of 3-manifolds obtained by the three topological transformations studied in \Cref{sec:topology}: null-homologous surgeries, negative-definite cobordisms, and clasper surgeries. Along the way in \Cref{thm:d-torelli} we establish a relationship between $d$-invariants and the Torelli group, so as to study the behavior of clasper surgeries. First, the $d$-invariants are governed by Spin$^c$ structures
 on 3- and 4-manifolds, so we begin with a review of those in \Cref{sec:spinc}.  We then give the key properties and computations for the Ozsv\'ath-Szab\'o $d$-invariants \cite{AbsGraded} in \Cref{sec:d-basic}.  In \Cref{sec:d-surgery} we give some basic constraints on the $d$-invariants of surgeries on null-homologous knots in rational homology spheres - \Cref{prop:d-surgery} below.  (This is surely known to experts but we could not find it in the literature.)  Finally, we prove \Cref{thm:d-torelli} in \Cref{sec:d-torelli}. Those familiar with the basics of $d$-invariants  may wish to begin at \Cref{sec:d-surgery}.

\subsection{Spin$^c$ structures in low-dimensions}\label{sec:spinc} For $Y$ a closed, connected, oriented 3-manifold $Y$, a result of Turaev \cite{Turaev} states that $\spinc$ structures on $Y$ can be thought of as equivalence classes of non-vanishing vector fields where two vector fields are equivalent if they are homotopic in the complement of a finite union of balls. This geometric approach to $\spinc$ structures is by no means their only definition, but with the intention of keeping this section user-friendly, we avoid an in-depth discussion of the definition(s) and focus instead on the properties that are key to establish our results. One of these relevant properties is the relationship between $\spinc(Y)$ and $H^2(Y;\Z)$. Taking the first Chern class of a $\spinc$ structure gives a map $$c_1:\spinc(Y)\to H^2(Y;\Z).$$ The image of $c_1$ is precisely $2H^2(Y;\Z)$. A Spin$^c$ structure $\mfs$ is said to be {\em torsion} if $c_1(\mfs)$ is a torsion element of $H^2(Y;\Z)$. A Spin$^c$ 3-manifold is a pair $(Y,\mfs)$. 

Given two 3-manifolds $Y_0$ and $Y_1$ and $\mfs_i\in\spinc(Y_i)$ for $i=0,1$, using the description of $\spinc$ structures as vector fields it is easy to see that there is a well defined element $\mfs_0\#\mfs_1\in\spinc(Y_0\#Y_1)$. Further, $c_1(\mfs \# \mfs')$ is identified with $c_1(\mfs) \oplus c_1(\mfs')$ under the isomorphism $H^2(Y \# Y';\Z) \cong H^2(Y;\Z) \oplus H^2(Y';\Z)$.  Hence, $\mfs \# \mfs'$ is torsion if and only if both $\mfs, \mfs'$ are.\\

Next, let $X$ be a compact, connected, orientable 4-manifold. Recall that an element of $H^2(X;\Z)$ is characteristic if its pairing with each element $\alpha \in H_2(X;\Z)$ agrees mod 2 with the self-intersection of $\alpha$. Similar to the 3-dimensional case, first Chern classes give a map $$c_1:\spinc(X)\to H^2(X;\Z).$$ In this case, the image of $c_1$ is in bijective correspondence with the set of {\em characteristic} elements of $H^2(X;\Z)$. Moreover, given a Spin$^c$ structure $\mft$ on a 4-manifold $X$, there is a restricted Spin$^c$ structure on $\partial X$ and this is natural with respect to the first Chern class, that is, $$c_1(\mft \mid_{\partial X}) = c_1(\mft) \mid_{\partial X}.$$ 

Finally, we remark that Spin$^c$ structures are in fact affine over the second cohomology group.  This implies that for a Spin$^c$ structure $\mathfrak{\nu}$ on a 3- or 4-manifold $M$, we obtain a different Spin$^c$ structure $\mathfrak{\nu} + \alpha$ for $\alpha \in H^2(M;\Z)$.  This affine structure is also natural with regards to restriction to the boundary and satisfies $c_1(\mathfrak{\nu} + \alpha) = c_1(\mathfrak{\nu}) + 2 \alpha$.

\subsection{Review of properties of $d$-invariants}\label{sec:d-basic}
We begin with the simplest version of correction terms or $d$-invariants,   introduced by Ozsv\'ath-Szab\'o in \cite{AbsGraded}. For a Spin$^c$ structure $\mfs$ on a rational homology sphere $Y$, the correction term $d(Y,\mfs)$ is a rational number, and is an invariant of Spin$^c$ rational homology cobordism.  More precisely, if $W$ is a cobordism from $Y_0$ to $Y_1$ with the same rational homology as $[0,1] \times S^3$,  then for any  Spin$^c$ structure $\mft$ on $W$ one has $d(Y_0,\mft|_{Y_0})=d(Y_1,\mft|_{Y_1})$.

\begin{theorem}[Theorem 1.2 and Theorem 9.6 \cite{AbsGraded}]\label{thm:d-properties}
Let $(Y,\mfs)$ and $(Y',\mfs')$ be Spin$^c$ rational homology spheres.  Then:
\begin{enumerate}
\item \label{d:orientation} $d(-Y,\mfs) = -d(Y,\mfs)$ 
\item \label{d:sum} $d(Y \# Y', \mfs \# \mfs') = d(Y,\mfs) + d(Y',\mfs')$
\item \label{d:neg-def} If $(Y,\mfs)$ bounds a negative-definite Spin$^c$ 4-manifold $(W,\mft)$, $$d(Y,\mfs) \geq \frac{c_1(\mft)^2 + b_2(W)}{4}.$$  If $(W,\mft) : (Y,\mfs) \to (Y',\mfs')$ is a negative-definite Spin$^c$ cobordism, then $$d(Y',\mfs') - d(Y,\mfs) \geq \frac{c_1(\mft)^2 - 3\sigma(W) - 2\chi(W)}{4}.$$
\end{enumerate}
\end{theorem}
Note that the second claim in \Cref{thm:d-properties}\eqref{d:neg-def} follows from the first by applying \Cref{thm:d-properties}\eqref{d:orientation} and \Cref{thm:d-properties}\eqref{d:sum}.  Also, since homology spheres have a single Spin$^c$ structure, we omit this from the notation when we work with such manifolds.\\

\begin{corollary}~\label{cor:boundingmeans0}
Let $Y$ be a rational homology sphere that bounds a rational homology ball.   Then
\[ d_{\max}(Y):= \max\{ d(Y, \mfs): \mfs \in \text{Spin}^c(Y)\} \geq 0\]
\end{corollary}
\begin{proof}
Let $\mathfrak{t}$ be a Spin$^c$ structure on the hypothesized rational homology ball $W$.  By \Cref{thm:d-properties}\eqref{d:neg-def},  we have that 
\[ d_{\max}(Y) \geq d(Y, \mathfrak{t}|_Y) \geq 0.\]
\end{proof}

As mentioned in the introduction, our arguments rely on the computation of the $d$-invariants for surgeries along knots of the form $T_{2,2k+1} \# T_{2,2k+1}$, so this is as good of a place as any to include this computation.

\begin{theorem}[Corollary 1.5, \cite{HFAlternating}]~\label{thm:surgeryalt}
Let $K$ be an alternating knot.  Then, 
\[
d(S^3_{1}(K)) = 2 min \{0, - \lceil -\sigma(K) / 4 \rceil \}.
\]  
Consequently, 
$$d(S^3_{1}(T_{2,2k+1} \# T_{2,2k+1})) =\begin{cases}
-2k & k\geq 0\\
\phantom{-2}0 & k<0.
\end{cases}
$$
\end{theorem}

\subsection{$d$-invariants of surgeries}\label{sec:d-surgery}
The goal of this subsection is to prove \Cref{prop:d-surgery}, which is concerned with the $d$-invariants of non-zero surgeries on null-homologous knots in rational homology spheres. Our analysis will require the study of correction terms for 3-manifolds $N$ with $b_1(N)>0$. An issue thus arises since $d$-invariants are not defined for Spin$^c$ structure in a general 3-manifold. However, if $\mfs$ is a torsion Spin$^c$ structure, i.e. $c_1(\mfs)$ is torsion in $H^2(Y)$, and $(Y, \mfs)$ has {\em standard} $HF^\infty$, then Ozsv\'ath-Szab\'o defined a correction term $d_b(Y,\mfs)\in\Q$ that is again an invariant of $\spinc$ rational homology cobordism\cite[Section 9]{AbsGraded}. Having standard $HF^\infty$ means that $HF^\infty(Y,\mfs) \cong HF^\infty(\#_{b_1(Y)} S^2 \times S^1, \mfs_0)$, where $\mfs_0$ is the unique torsion Spin$^c$ structure on $\#_{b_1(Y)} S^2 \times S^1$. We remark that if $Y$ is a rational homology sphere and $\mfs$ is any Spin$^c$ structure on $Y$, then $(Y, \mfs)$ has standard $HF^{\infty}$ \cite[Theorem 10.1]{HFPA}.  Similar to \Cref{thm:d-properties}, we have the following: 

\begin{theorem}[Theorem 1.5, Proposition 2.5, and Theorem 10.1 in \cite{HFPA}, Theorem 9.15 in \cite{AbsGraded}, Proposition 4.3 in \cite{LR}, Proposition 4.0.5 in \cite{KPThesis}]\label{thm:d-properties-general}
Let $(Y,\mfs)$ and $(Y',\mfs')$ be 3-manifolds with standard $HF^\infty$ equipped with torsion Spin$^c$ structures.  Then:
\begin{enumerate}
\item \label{d:QHS3} if $Y$ is a rational homology sphere,  then $d(Y,\mfs) = d_b(Y,\mfs)$. 
\item \label{d:orientation-general} $(-Y,\mfs)$ has standard $HF^\infty$, but $d_b(-Y,\mfs)$ and $d_b(Y,\mfs)$ are not necessarily related.
\item \label{d:connect-sum-general} $(Y \# Y',\mfs \# \mfs')$ has standard $HF^\infty$ and $d_{b}(Y \# Y', \mfs \# \mfs') = d_{b}(Y,\mfs) + d_{b}(Y',\mfs')$.    
\item \label{d:negative-definite-general} If $(Y,\mfs)$ is the boundary of a negative semidefinite Spin$^c$ 4-manifold $(W,\mft)$ with $b_1(W) = 0$, then
 $$d_{b}(Y,\mfs) \geq \frac{c_1(\mft)^2 + b_2^-(W) - 2b_1(Y)}{4}.$$
\item \label{d:circle-bundles} If $Y$ is a circle bundle over a genus $g$ surface with Euler number $n$, and $\mfs$ is any torsion Spin$^c$ structure, then $(Y,\mfs)$ has standard $HF^\infty$ if and only if $n \neq 0$ or $g = 0$.  
\end{enumerate}
\end{theorem}

Recall that a 4-manifold is {\em negative semidefinite} if the intersection form is negative definite on the non-degenerate part. The primary example is any 4-manifold with vanishing intersection form. For the readers familiar with other notions of $d$-invariants we remark that $d_b(-Y,\mfs)=-d_{top}(Y,\mfs)$, where $d_{top}$ is the ``topmost'' $d$-invariant --see for example \cite{LR}.  

In order to proof \Cref{prop:d-surgery}, we need we will need to guarantee that the $\spinc$ structures on the 3-manifolds under consideration are torsion. The following homological lemma accomplishes this.  

\begin{lemma}\label{lem:tors-spinc}
Let $(X,\mft)$ be a $\spinc$ 4-manifold and $S_g$ an orientable embedded surface of genus $g$ and self-intersection $n$.  Let $N_{n,g}$ be the boundary of a tubular neighborhood of $S_g$, i.e. a circle bundle over $S_g$ with Euler number $n$. 
If either $n\neq 0$ or $\langle c_1(\mft), [S_g] \rangle_X = 0$,  then $\mft \mid_{N_{n,g}}$ is torsion.
\end{lemma}

\begin{proof}
Recall that a $\spinc$ structure $\mfs$ on $N_{n.g}$ is torsion if $c_1(\mfs)$ is a torsion element of $H^2(N_{n,g};\Z)$, and that $$H_2(N_{n,g};\Z)=\begin{cases}
\Z^{2g} & \text{if } n\neq 0\\
\Z^{2g+1} & \text{if } n=0
\end{cases},$$
generated entirely by the tori $T_\gamma$ obtained from restricting the circle bundle to loops $\gamma$ in $S_g$ whenever $n\neq 0$, and these tori plus the homology class of $S_g$ when $n=0$.

Thus to show that $\mft \mid_{N_{n,g}}$ is torsion, it suffices to show that $c_1(\mft \mid_{N_{n,g}})$ evaluates to be trivial on the generators of $H_2(N_{n,g};\Z)$. By the naturality of $c_1$ for $\spinc$ structures we have $$\langle c_1(\mft\mid_{N_{n,g}}), [T_\gamma] \rangle_{N_{n,g}} =\langle c_1(\mft)_{N_{n,g}}, [T_\gamma] \rangle_{N_{n,g}} = \langle c_1(\mft), [T_\gamma] \rangle_X.$$
However, in a neighborhood of $S_g$ in $X$ one can extend the circle fibers over the disk fiber, and so $T_\gamma$ bounds a solid torus. Consequently $\langle c_1(\mft), [T_\gamma] \rangle_X = 0$, and this completes the proof.
\end{proof}

To establish notation, given a null-homologous knot $K$ in a 3-manifold $Y$, let 
\begin{equation}\label{2-handlecob}
W_n(K)=[0,1]\times Y \bigcup D^2\times D^2 
\end{equation}
denote the 2-handle cobordism from $Y$ to $Y_n(K)$. Note that $\spinc(Y_n(K)) = \spinc(Y) \oplus \Z/n$.  Every $\spinc$ structure on $Y_n(K)$ extends over $W_n(K)$ and the operator of extend over $W_n(K)$ and restrict to $Y$ induces projection from $\spinc(Y_n(K))$ to $\spinc(Y)$ in the splitting above.  See \cite{HFKZ} for more details.

\begin{proposition}\label{prop:d-surgery}
Let $n\neq 0$ be an integer and $g \geq 0$.  There exists a constant $C_{g,n}$ that depends on only $|n|$ and $g$ with the following property.
For any null-homologous knot $K$ that bounds a genus $g$ surface in a rational homology sphere $Y$,  
and any $\mfs \in \spinc(Y)$ and $\mfs_n \in \spinc(Y_n(K))$ that are cobordant over $W_n(K)$,  we have 
\[|d(Y_n(K),\mfs_n) - d(Y,\mfs)| \leq C_{g,n}.\] 
\end{proposition}

Before proving the proposition, a few remarks are in order: 
\begin{remark}\phantom{.}

\begin{itemize}
\item The distinction between $|n|$ and $n$ is purely cosmetic, since we can take the max of $C_{g,n}$ and $C_{g,-n}$.  The advantage of stating it this way is simply for ease of exposition later.
\item This result was proved by Fr\o yshov for $\pm 1$-framed surgeries for the instanton $h$-invariant (see \cite[Corollary 1]{Froyshov}).  
\item For knots in $S^3$, the result follows easily from the Ni-Wu $d$-invariant surgery formula \cite{NiWu}.  The arguments there can be generalized to prove \Cref{prop:d-surgery} if one allows the constant to depend on the ambient manifold $Y$ as well. 
\item There are analogues of \Cref{prop:d-surgery} for more general cobordisms between 3-manifolds, but we do not state or prove those here for ease of exposition.
\end{itemize}
\end{remark}

\begin{proof}[Proof of \Cref{prop:d-surgery}] Consider a rational homology sphere $Y$ and $K\subset Y$ a null-homologous knot. Let $W=W_n(K)$ be as in \Cref{2-handlecob}. Assume for simplicity that $n<0$.

To get a lower bound on $d(Y_n(K), \mfs_n) - d(Y, \mfs)$, notice that $W$ is a negative definite manifold with oriented boundary $-Y \coprod Y_n(K)$. By \Cref{thm:d-properties}\eqref{d:negative-definite-general}
 $$d(Y_n(K), \mfs_n) - d(Y, \mfs)\geq\frac{c_1(\mft)^2+1}{4}$$
for any Spin$^c$ structure $\mft$ on $W$ which restricts to $\mfs$ on $Y$ and $\mfs_n$ on $Y_n(K)$. We would like to choose $\mft$ giving a lower bound on $c_1(\mft)^2$ independent of $Y, K, \mfs, \mfs_n$.  Let $[S]$ be a generator of $H^2(W)$, and note that $(k[S])^2 = \frac{k^2}{n}$.  By naturality of the affine $H^2$ structure for Spin$^c$ structures, observe that $\mft$ and $\mft + n[S]$ restrict to the same Spin$^c$ structure on both $Y$ and $Y_n(K)$.  Therefore, by adjusting with copies of $[S]$, we can find $\mft$ on $W$ restricting to $\mfs$ and $\mfs_n$ such that $|c_1(\mft)|^2 \leq 4|n|$.  Thus, we obtain a lower bound on $d(Y_n(K), \mfs_n) - d(Y, \mfs)$, which depends only on $n$.

Unfortunately $-W$ is not negative definite, and so we cannot simply reverse orientations to obtain a lower bound. Luckily,  if we excise the homology from $-W$, then the resulting manifold is negative semidefinite and we can obtain a similar bound. However, this forces us to consider a correction term determined by the particular representative of homology we excise. To be more precise, for $F$ a genus $g$ Seifert surface for $K$ in $Y$, denote by $\widehat{F}_g$ the closed surface obtained by pushing $\{1\}\times F$ into the interior of $[0,1]\times Y\subset W$, and capping it off with the core of the 2-handle. Notice that $H_2(W;\Z)$ is generated by the homology class of $\widehat{F}_g$, so `excising' second homology from $W$ amounts to removing a tubular neighborhood of $\widehat{F}_g$. To this end construct a 4-manifold $Z$ by taking the union of collar neighborhoods of $Y,\,Y_n(K)$ in $W$, and tubular neighborhoods of arcs joining $Y,\,Y_n(K)$ to $\widehat{F}_g$ (see \Cref{tubing} for a schematic diagram of this construction). \\
\begin{figure}[h!]
\centering
\def\svgwidth{0.75\textwidth}
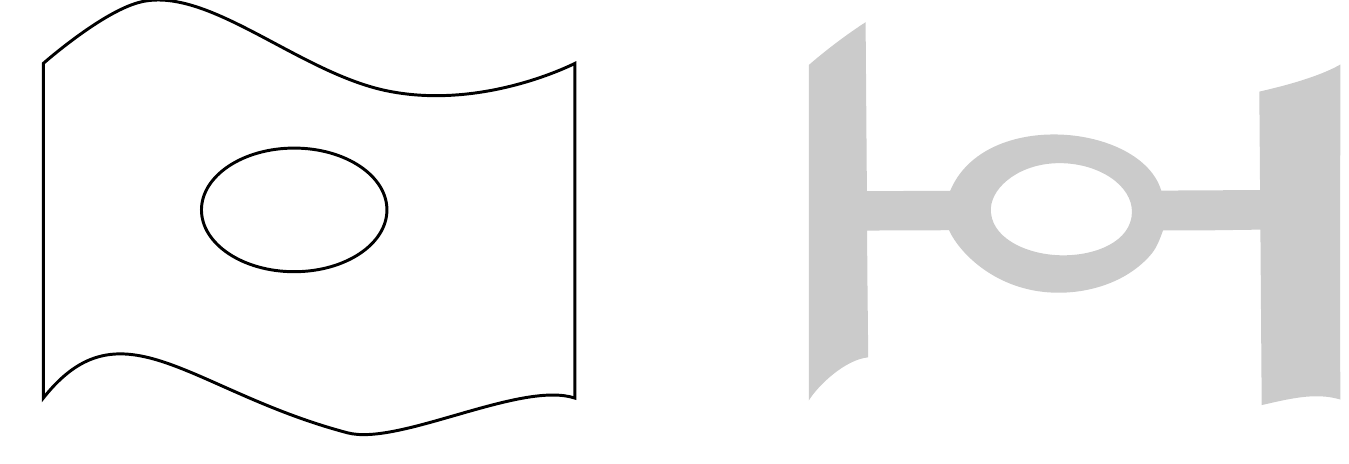
\caption{The surface $\widehat{F}_g$ in $W_n(K)$ (left) and a cobordism $Z$ from $Y\sqcup -Y_n(K)$ to $Y \# N_{n,g} \# -Y_n(K)$ (right).}\label{tubing}
\end{figure}

We will obtain a lower bound for $d(Y,\mfs) - d(Y_n(K),\mfs_n)$ from \Cref{thm:d-properties-general}\eqref{d:negative-definite-general} applied to the manifold $X=-W\setminus -Z$. In order to do so, one has to understand the boundary of $X$, its intersection form, and  the restriction of elements of Spin$^c(X)$ to $\partial X$. With this in mind, denote by $N_{n,g}$ the boundary of a neighborhood of $\widehat{F}_g$, namely a circle bundle of Euler number $n$ over a genus $g$ surface. We then have 
$$\partial Z=-Y\sqcup \left(Y \# N_{n,g} \# -Y_n(K)\right) \sqcup Y_n(K), \text{ so that } \partial X=Y \# N_{n,g} \# -Y_n(K).$$
Next, using the Mayer-Vietoris sequence for the decomposition $-W=X\underset{\partial X}{\cup} -Z$, we get that the inclusion of $\partial X$ into $X$ induces an isomorphism on rational second homology and that $b_1(X) = 0$. In addition, since any surface in $X$ can be pushed off itself in a collar of $\partial X$, the intersection form on $X$ vanishes. It follows that $X$ is a negative semidefinite manifold (as well as a positive semidefinite one).  
Finally, we study the Spin$^c$ structures on $X$ so that we can invoke \Cref{thm:d-properties-general}\eqref{d:negative-definite-general}. By assumption, $\mfs \in \spinc(Y)$ and $\mfs_n \in \spinc(Y_n(K))$ are cobordant over $W$, that is, there exists a $\spinc$ structure $\mft$ on $W$ which restricts to $\mfs_n$ and $\mfs$ on the ends. Thus, $(Y \# N_{n,g} \# -Y_n(K), \mfs \# \nu  \# \mfs_n)$ is the $\spinc$ boundary of $(X,\mft|_{X})$.

To complete the proof, notice first that a combination of \Cref{thm:d-properties-general}\eqref{d:negative-definite-general} - \eqref{d:circle-bundles} gives $$d(Y,\mfs)+d_{b}(N_{n,g},\nu)-d(Y_n(K),\mfs_n) \geq \frac{c_1(\mft)^2 + b_2^-(X) - 2b_1(\partial X)}{4}=\frac{0 + 0 - 4g}{4}=-g.$$ To see this, note that the intersection form of $X$ vanishes, so $c_1(\mft)^2 = b_2^-(X) = 0$. Since there are only finitely many torsion Spin$^c$ structures on $N_{n,g}$, the maximum value of $|d_b(N_{n,g},\nu)|$ over the torsion elements of $\spinc(N_{n,g})$ is an invariant of $N_{n,g}$ and hence only depends on $n$ and $g$. We thus get the lower bound for $d(Y,\mfs) - d(Y_n(K),\mfs_n)$ sought.
\end{proof}

\subsection{$d$-invariants and the Torelli group}\label{sec:d-torelli}
In this subsection we prove \Cref{thm:d-torelli} regarding the effect Torelli surgeries have on $d$-invariants. To begin, let $\Sigma_g$ be a fixed genus $g$ surface. If $\mathcal{M}_g$ denotes the mapping class group of $\Sigma_g$, then the Torelli group $\mathcal{T}_g$ is the subgroup of $\mathcal{M}_g$ consisting of elements $\phi:\Sigma_g\to\Sigma_g$ for which the induced map $\phi_*:H_1(\Sigma_g;\Z)\to H_1(\Sigma_g;\Z)$ is the identity.  Note that $\mathcal{T}_1$ is trivial,  $\mathcal{T}_2$ is generated by Dehn twists along separating curves,  and $\mathcal{T}_g$  is generated by bounding pair maps for $g \geq 3$,  as follows from~\cite{Birman,Johnson,Powell}. 
These last maps are defined as follows: consider a pair of non-separating curves $\alpha_\pm$ on $\Sigma_g$ that separates $\Sigma_g$ into two subsurfaces having $\alpha_-\sqcup\alpha_+$ as their common (unoriented) boundary. A bounding pair map consists of the composition of a positive Dehn twist along $\alpha_+$ and a negative Dehn twist along $\alpha_-$.  For more information on the above, see \cite{HatcherMargalit}.

Let $S$ be an embedded (parameterized) copy of $\Sigma_g$ inside of a 3-manifold $Y$.  For $\phi \in \mathcal{T}_g$, let $Y_\phi$ denote the result of removing $S$ and gluing it back using $\phi$.  Since $\phi$ is in the Torelli group, $Y_\phi$ and $Y$ have the same homology. The following theorem is the main result of this subsection. 

\begin{reptheorem}{thm:d-torelli}
Let $\phi$ be an element of the Torelli group for a closed surface $\Sigma_g$.  Express $\phi$ as a product of $k$ separating Dehn twists and/or bounding pair maps.  There exists a constant $C_{k,g}$ depending only on the length of this word and the genus of the surface, with the following property.  If $Y$ is any homology sphere and $Y_\phi$ is the result of removing an embedded (parameterized) genus $g$ surface in $Y$ and regluing by $\phi$, then 
\[
|d(Y) - d(Y_\phi)| \leq C_{k,g}.
\]
\end{reptheorem}

Since a clasper surgery is a Torelli surgery on a genus 3 handlebody always with the same monodromy, we have:
\begin{corollary}\label{cor:d-clasper}
There exists a constant $C$ such that if $M$ is an integer homology sphere obtained by clasper surgery on another integer homology sphere $N$, then $|d(M) - d(N)| \leq C$.  
\end{corollary}

The proof of \Cref{thm:d-torelli} is similar to that of \Cref{prop:d-surgery}.  In particular, for a separating Dehn twist (which simply corresponds to surgery along a null-homologous knot) or bounding pair map, we consider the associated 2-handle cobordism and use this to constrain the $d$-invariants.  To handle the case of bounding pair maps, we will need the {\em twisted} $d$-invariant due to Behrens-Golla \cite{BehrensGolla}, an even more general notion of $d$-invariant  that works in the absence of standard $HF^\infty$. (See also the work of Levine-Ruberman \cite{LevineRuberman}.) For a torsion $\spinc$ structure $\mfs$ on a 3-manifold $Y$, they define an invariant $\dtwist(Y,\mfs) \in \mathbb{Q}$ (not to be confused with the Hendricks-Manolescu ``$d$-lower'') which satisfies the following properties similar to $d$ and $d_b$.

\begin{theorem}[Behrens-Golla \cite{BehrensGolla}]\label{thm:d-twist}  
Let $Y$ be a closed, connected, oriented 3-manifold and $\mfs$ a torsion $\spinc$ structure on $Y$.  
\begin{enumerate}
\item if $Y$ is a rational homology sphere, $\dtwist(Y,\mfs) = d(Y,\mfs)$;
\item $\dtwist(Y \# Y', \mfs \# \mfs') = \dtwist(Y,\mfs) + d(Y',\mfs')$ for any $\spinc$ rational homology sphere $Y'$;
\item\label{eq:d-twist-semidef} if $(W, \mft)$ is a negative semidefinite Spin$^c$ 4-manifold with boundary $(Y,\mfs)$, then 
$$ \dtwist(Y,\mfs) \geq \frac{c^2_1(\mft) + b_2^-(W) - 2b_1(Y)}{4}.$$
\end{enumerate}  
\end{theorem}

\begin{proof}[Proof of \Cref{thm:d-torelli}]
First, suppose that $\phi$ is a single positive (respectively negative) Dehn twist along a curve $\alpha$ which separates $\Sigma_g$ into two subsurfaces $S_1,S_2$ each bounded by $\alpha$. Then, the image of $\alpha$ in $Y$ is a null-homologous knot $K$ of genus at most $g$ and $Y_\phi$ is obtained by $-1$- (respectively $+1$-) Dehn surgery along $K$. We obtain a bound $|d(Y) - d(Y_\phi)| \leq C_{g,1}$ from \Cref{prop:d-surgery}, where $C_{g,1}$ is independent of the sign of the twist or particular separating curve (but does depend on $g$). Next, suppose that $\phi$ is a composition of Dehn twists along separating curves $\alpha_1,\ldots,\alpha_k$. Then, we can view $Y_\phi$ as a sequence of $\pm 1$-surgeries along the images of $\alpha_1,\ldots,\alpha_k$ in $Y$ \cite[pp. 539-540]{lickorish}.  Repeated application of \Cref{prop:d-surgery} shows that $|d(Y) - d(Y_\phi)| \leq k C_{g,1}$ for some constant that depends only on $g$.

To finish the proof it therefore suffices to show that there is a constant $A_g$ such that if $\phi$ is a bounding pair map, then $|d(Y) - d(Y_\phi)| \leq A_g$, where again $A_g$ is independent of $Y$.  Let $\alpha_\pm$ be two curves in $\Sigma_g$ which cobound a subsurface of $\Sigma_g$. Assume that $\phi$ is the composition of a positive Dehn twist along $\alpha_+$ and a negative Dehn twist along $\alpha_-$. Let $K_\pm$ be the (oriented) images of $\alpha_\pm$ in $Y$ and let $F$ be a subsurface of $S$ in $Y$ which $K_+$ and $K_-$ cobound.  In this case, $Y_\phi$ is obtained from $Y$ by doing surgery on the link $K_+ \cup K_-$.  Let $W$ be the 2-handle cobordism from $Y$ to $Y_\phi$.  Unfortunately, we cannot use $S$ to constrain the genera of either $K_+$ or $K_-$ like in the separating case, because $\alpha_\pm$ are assumed to be non-separating in $S$. Nevertheless, we can find a primitive, homologically essential, square zero, closed, and oriented surface in $W$ obtained by the union of $F$ and one core disk for each of the 2-handles, with orientations determined by those of $K_\pm$.  Call this surface $\widehat{F}$.  Note that it has the same genus as $F$, which is at most $g$. 

We now argue that $\widehat{F}$ has square zero in $W$.  First, observe that an $F$-framed pushoff of $K_+$ is homologous to $K_-$ in the exterior of $K_+ \cup K_-$ and vice versa. It follows that 
\[ lk(K_+, K_-)= -\lk(K_+,  K_+^F)\]
where $K_+^F$ refers to a push-off of $K_+$ in $F$. (The seemingly opposite signs in the above formula come from the fact that $K_+ \cup -K_-$ bounds.)

Cutting and regluing by the positive Dehn twist along $\alpha_+$ corresponds to $(-1)$-surgery on $\alpha_+$, where these surgery coefficients are relative to the framings induced by $F$, not the Seifert framing \cite[pp. 539-540]{lickorish}. More precisely, the surgery framing curve is homologous to $-\mu_{K_+}+K^F_{+}$. A similar argument shows the negative Dehn twist along $\alpha_-$ corresponds to $(+1)$-surgery along $\alpha_-$,  again relative to the framing induced by $F$. Therefore, relative to the Seifert framings, the linking matrix for the surgery on $K_+ \cup K_-$ is given by 
\[
\begin{pmatrix} -\ell - 1 & \ell \\ \ell & -\ell + 1 \end{pmatrix}
\]
where $\ell = lk(K_+,K_-)$.  This describes a natural basis for $H_2(W)$, such that $\widehat{F}$ corresponds to $(1,1)$, and we see that $\widehat{F}$ is square zero.  

As in the proof of \Cref{prop:d-surgery}, we break our cobordism $W$ into two pieces.  Let $Z$ be the union of collar neighborhoods of $Y_\phi$ and $Y$,  a neighborhood of $\widehat{F}$,  and neighborhoods of arcs running from $Y_\phi$ and $Y$ to $\widehat{F}$; and let $X:= W\smallsetminus Z$. 
If there existed a surface $\Sigma$ in $X$ with nonzero self-intersection,  then $\Sigma,  \widehat{F}$ would be linearly independent and hence form a basis for the rational second homology of $W$.  With respect to this basis, the intersection form would be given by $\begin{pmatrix} * & 0 \\ 0 & 0 \end{pmatrix}$,  which would contradict the nondegeneracy of the intersection form on $W$,  a 4-manifold with integer homology sphere boundary components.  Therefore,   $X$ has vanishing intersection form and in particular is a negative semidefinite 4-manifold with boundary $-Y \# Y_\phi \# \widehat{F} \times S^1$. We would hope for an inequality of the $d$-invariants, as with null-homologous surgeries. The problem is that  $\widehat{F} \times S^1$ does {\em not} have standard $HF^\infty$ when the genus of $\widehat{F}$ is at least 1, so we must use the twisted correction terms.  

We would like to invoke \Cref{thm:d-twist}\eqref{eq:d-twist-semidef}.  Note that the homology spheres $Y$ and $Y_\phi$ have unique $\spinc$ structures and $\widehat{F} \times S^1$ has a unique torsion $\spinc$ structure since $H^2(\widehat{F} \times S^1)$ is torsion-free.  Thus, to utilize \Cref{thm:d-twist}\eqref{eq:d-twist-semidef}, we just need to show that there exists a $\spinc$ structure on $X$ that restricts to be torsion on $\widehat{F} \times S^1$. By \Cref{lem:tors-spinc}, it suffices to construct a $\spinc$ structure on $X$ whose first Chern class evaluates to be trivial on $\widehat{F}$.  We will construct our $\spinc$ structure on $W$ and then restrict to $X$.  (Note that $\spinc$ structures on a closed 3-manifold are the same as on once-punctured 3-manifolds, so this restriction will not cause any issues.)  

If $\ell$ is odd, then $W$ has even intersection form, and so we can choose the unique $\spinc$ structure $\mft$ on $W$ which satisfies $c_1(\mft) = 0$; this necessarily restricts to $c_1 = 0$ on $\widehat{F} \times S^1$, which implies that the restriction of $\mft$ to $X$ is torsion. The case that $\ell$ is even goes as follows. A $\spinc$ structure on $\widehat{F} \times S^1$ is torsion if and only if $c_1$ pairs trivially with $\widehat{F}$.  Let $a \in H^2(W) = H^2(W,\partial W)$ be the element which pairs to $\pm 1$ on $T_\pm$ where $T_\pm$ are generators of $H_2(W)$ obtained by capping off Seifert surfaces for $K_\pm$ in $Y$ with the cores of the corresponding 2-handle.  One can quickly check that $a$ is characteristic, and therefore $a=c_1(\mft)$ for some $\spinc$ structure $\mft$ on $W$. Further,  $a= c_1(\mft)$ pairs to zero with $\widehat{F}$, since $[\widehat{F}] = T_+ + T_-$.  In conclusion, independent of $\ell$, there is a $\spinc$ structure on $W$ (and hence $X$) that restricts to be torsion on $\widehat{F} \times S^1$.    

By the above discussion, \Cref{thm:d-twist} \eqref{eq:d-twist-semidef} gives the bound
\[
\dtwist(-Y \# Y_\phi \# \widehat{F} \times S^1) \geq \frac{-2b_1(\partial X)}{4} = -\frac{2g+1}{2}, 
\]
where we do not write any Spin$^c$ structures because each summand has a unique torsion one.  (There is no $c_1^2$ or $b_2^-$ term because the intersection form on $X$ vanishes.) We thus obtain a lower bound for $\dtwist(-Y \# Y_\phi \# \widehat{F} \times S^1)$ that depends only on the genus of $F$.  By \Cref{thm:d-twist}, 
\[
\dtwist(-Y \# Y_\phi \# \widehat{F} \times S^1) = -d(Y) + d(Y_\phi) + \dtwist(\widehat{F} \times S^1).  
\]
Therefore, we get $d(Y) - d(Y_\phi) \leq C$, where $C$ depends only on the genus of the subsurface $F$.  Since the genus of $F$ is at most $g$, we obtain a lower bound independent of the subsurface.  By reversing orientations and repeating the process, we gain the opposite inequality $d(Y_\phi) - d(Y) \leq C$. This completes the proof.  
\end{proof}

\begin{remark}
The proof shows something slightly stronger than \Cref{thm:d-torelli}.  Namely, the constant can be refined further if one knows the genera of the subsurfaces involved in the factorization into separating Dehn twists and bounding pair maps.
\end{remark}

%% file: tubing.pdf_tex
\begingroup%
  \makeatletter%
  \providecommand\color[2][]{%
    \errmessage{(Inkscape) Color is used for the text in Inkscape, but the package 'color.sty' is not loaded}%
    \renewcommand\color[2][]{}%
  }%
  \providecommand\transparent[1]{%
    \errmessage{(Inkscape) Transparency is used (non-zero) for the text in Inkscape, but the package 'transparent.sty' is not loaded}%
    \renewcommand\transparent[1]{}%
  }%
  \providecommand\rotatebox[2]{#2}%
  \newcommand*\fsize{\dimexpr\f@size pt\relax}%
  \newcommand*\lineheight[1]{\fontsize{\fsize}{#1\fsize}\selectfont}%
  \ifx\svgwidth\undefined%
    \setlength{\unitlength}{388.4606809bp}%
    \ifx\svgscale\undefined%
      \relax%
    \else%
      \setlength{\unitlength}{\unitlength * \real{\svgscale}}%
    \fi%
  \else%
    \setlength{\unitlength}{\svgwidth}%
  \fi%
  \global\let\svgwidth\undefined%
  \global\let\svgscale\undefined%
  \makeatother%
  \begin{picture}(1,0.34380297)%
    \lineheight{1}%
    \setlength\tabcolsep{0pt}%
    \put(0,0){\includegraphics[width=\unitlength,page=1]{tubing.pdf}}%
    \put(0.94165679,0.10697396){\color[rgb]{0,0,0}\makebox(0,0)[lt]{\lineheight{1.25}\smash{\begin{tabular}[t]{l}$Z$\end{tabular}}}}%
    \put(0.57273377,0.00582978){\color[rgb]{0,0,0}\makebox(0,0)[lt]{\lineheight{1.25}\smash{\begin{tabular}[t]{l}$-Y$\end{tabular}}}}%
    \put(0.99628334,0.00621443){\color[rgb]{0,0,0}\makebox(0,0)[rt]{\lineheight{1.25}\smash{\begin{tabular}[t]{r}$Y_n(K)$\end{tabular}}}}%
    \put(-0.00233795,0.00582978){\color[rgb]{0,0,0}\makebox(0,0)[lt]{\lineheight{1.25}\smash{\begin{tabular}[t]{l}$-Y$\end{tabular}}}}%
    \put(0.42667321,0.00621443){\color[rgb]{0,0,0}\makebox(0,0)[rt]{\lineheight{1.25}\smash{\begin{tabular}[t]{r}$Y_n(K)$\end{tabular}}}}%
    \put(0.19353809,0.18333787){\color[rgb]{0,0,0}\makebox(0,0)[lt]{\lineheight{1.25}\smash{\begin{tabular}[t]{l}$\widehat{F}_g$\end{tabular}}}}%
    \put(0.23059987,0.08380505){\color[rgb]{0,0,0}\makebox(0,0)[lt]{\lineheight{1.25}\smash{\begin{tabular}[t]{l}$W_n(K)$\end{tabular}}}}%
    \put(0,0){\includegraphics[width=\unitlength,page=2]{tubing.pdf}}%
  \end{picture}%
\endgroup%

%% file: proof-of-theorem.tex
\subsection{The proof of \Cref{thm:mainnullhom}.}\label{sec:proofs}

We are now ready to prove the following result, which will quickly imply~\Cref{thm:mainnullhom}.

\begin{proposition}~\label{prop:boundingd}
Let $P$ be a pattern in the solid torus with winding number $w$. Suppose that there exists a prime power $q$ dividing $w$ such that, if $\eta_1, \dots, \eta_q$ denote the lifts of $\eta$ to $\Sigma_q(P(U)$, we have 
\begin{enumerate}
\item  $\lk(\eta_i,\eta_j) \geq 0$ for all $i \neq j$, but is not identically zero, and 
\item   $\eta_1,\ldots, \eta_q$ are null-homologous in $\Sigma_q(P(U))$. 
\end{enumerate}
Then there exists a constant $C_{P,q}$ depending only on $P$ and $q$ such that for any knot $K$, 
\[ d_{\max}(\Sigma_q(P(K)) \leq d(\Sigma_2(C_{2,1}(K)) + C_{P,q}.\]
\end{proposition}

\begin{proof}

\textit{Step 1: Converting $(\Sigma_q(P(U)), \cup_{i=1}^q \eta_i)$ to $(S^3,L)$ via null-homologous surgery.} Applying \Cref{prop:qhs3-to-s3} to $(\Sigma_q(P(U)),  \cup_{i=1}^q \eta_i)$ tells us that there exists (i) a framed link $L$ in $S^3$ with the same linking-framing information as $\cup_{i=1}^q \eta_i$, (ii) natural numbers $n$,  and (iii) a sequence $m_1, \dots,m_n$ of natural numbers such that $(S^3,L)$ can be transformed into  $(\Sigma_q(P(U)) \#_{i=1}^n L(m_i,1),  \cup_{i=1}^q \eta_i)$ via a  finite sequence of null-homologous integral surgeries along a link $\Gamma = \gamma_1 \cup \ldots \cup \gamma_k$.  For convenience in notation,  let $Y:= \#_{i=1}^n L(m_i,1)$ and define
\begin{align}\label{eqn:c0bound}
 c_0:= \max \left\{ \left| d \left( Y,  \mfs\right)\right|: \mfs \in \text{Spin}^c \left(Y \right) \right\}. 
 \end{align}
Since $\Gamma$ is disjoint from $L$, we can consider each component $\gamma$ of $\Gamma$ as a knot 
\[ \gamma \subset S^3 \smallsetminus \nu(L) \subset(S^3, L)*E_K.\]  Moreover,  since we are considering a sequence of null-homologous surgeries in $S^3 \smallsetminus L$,  each knot $\gamma_{i+1}$ bounds a surface in the result of surgery on $\gamma_1, \dots, \gamma_i$ in $(S^3, L)*E_K$ whose genus is bounded above independently of $K$ (i.e., only in terms of $P$ and $q$).  (See the discussion following \Cref{prop:qhs3-to-s3} for what is meant by a sequence of nullhomologous surgeries.)  Finally, surgery on $(S^3, L)*E_K$ along $\gamma_1, \dots, \gamma_k$ using the same coefficients as before transforms $(S^3,L)*E_K$ to
\begin{align*}
 (\Sigma_q(P(U)) \# Y,  \cup_{i=1}^q \eta_i) *E_K
 &=
\left[(\Sigma_q(P(U)),  \cup_{i=1}^q \eta_i)*E_K \right]  \# Y\\
&= \Sigma_q(P(K)) \# Y. 
\end{align*}

Let $W_K$ denote the 4-manifold obtained from attaching $2$-handles to $[0,1] \times \left((S^3, L)*E_K\right)$ along each $\{1\} \times \gamma$ for each component $\gamma$ of $\Gamma$ with the surgery framings,  and note that $W_K$ is the surgery 2-handle cobordism from $(S^3, L)*E_K$ to $\Sigma_q(P(K)) \# Y$. Let $\mfs_\Sigma$ denote a Spin$^c$ structure on $\Sigma_q(P(K))$ with maximal corresponding $d$-invariant, i.e.  such that 
\begin{align}\label{eqn:dmax}
 d(\Sigma_q(P(K)), \mfs_\Sigma)=d_{\max}(\Sigma_q(P(K)).\end{align}
Choose any Spin$^c$-structure $\mfs_Y$ on $Y$. As noted before~\Cref{prop:d-surgery},   there exists a Spin$^c$ structure $\mathfrak{t}$ on $W_K$ that restricts to $\mfs_\Sigma \# \mfs_Y$ on $\Sigma_q(P(K)) \# Y. $ Since $(S^3, L)*E_K$ is an integer homology sphere,  $\mathfrak{t}$ restricts to the unique Spin$^c$ structure on $(S^3, L)*E_K$. 

By repeated application of~\Cref{prop:d-surgery} there exists a constant $c_1$ depending only on the framing of the surgeries and the genera of the knots $\gamma_1, \dots, \gamma_k$ (in particular,  independent of $K$) such that 
\begin{align}\label{eqn:c1bound}
|d((S^3, L)*E_K) - d(\Sigma_q(P(K)) \# Y, \mfs_\Sigma\#\mfs_Y)| \leq c_1.
\end{align}
We can now use the additivity of $d$-invariants under connected sum \Cref{thm:d-properties}.\ref{d:sum}, together with \eqref{eqn:c0bound},  \eqref{eqn:dmax}, and \eqref{eqn:c1bound} above to obtain that
\begin{align}\label{eqn:bigbound1}
|d((S^3, L)*E_K) - d_{\max}(\Sigma_q(P(K)))| 
\leq  |d( Y,  \mfs_Y)|+c_1 \leq c_0+c_1. 
\end{align}

\textit{Step 2: Reducing the linking of $L$ via negative-definite cobordisms. }
By assumption, $\lk(\eta_i,\eta_j) \geq 0$ for all $i \neq j$ and $\lk(\eta_i,\eta_j)>0$ for some $i \neq j$.  By \Cref{lem:linkingframing},  we know that $\text{fr}(\eta_i)=-\sum_{j \neq i} \lk(\eta_i, \eta_j)$ for all $i$.  
Since $(S^3, L)$ is obtained from $(\Sigma_q(P(U)), \cup_{i=1}^q \eta_i)$ by null-homologous surgeries, which do not change linking-framing information,  the link $L$ also satisfies the hypotheses of \Cref{prop:reduce-linking}. 
Therefore,  the conclusion holds: there exists a negative definite cobordism $W$ from $(S^3, L)$ to $(S^3, L')$,  where $L'$ is a $q$-component link with the same linking-framing information as $H_q$.  Moreover, this cobordism is of a particularly simple type, consisting of attaching 2-handles to $I\times S^3$ along curves which are disjoint from $L$.  It follows that for any knot $K$,  by cutting out $I\times\nu(L)$ and gluing in $I\times \cup_{i=1}^q E_K$,  we obtain a negative definite cobordism from $(S^3, L)*E_K$ to $(S^3, L')*E_K$.  
By the behavior of $d$-invariants under negative-definite cobordisms (\Cref{thm:d-properties}.\ref{d:neg-def}) we therefore have that for any knot $K$, 
\begin{align}\label{eqn:c2bound}
d((S^3, L)*E_K) \leq d((S^3, L')*E_K).
\end{align}

\textit{Step 3: Converting $L'$ to $H_q$ via clasper moves. }
We now apply \Cref{prop:make-standard} to $(S^3, L')$ to observe that there exists some length $M$ sequence of clasper surgeries transforming $(S^3, L')$ into $(S^3, H_q)$,  where we remind the reader that $H_q$ is the split union of a $-1$-framed positive Hopf link $H$ with a $(q-2)$-component 0-framed unlink. 
 For any knot $K$,  this induces a length $M$ sequence of clasper surgeries transforming $(S^3,L')*E_K$ into $(S^3, H_q)* E_K$.  By repeated application of~\Cref{cor:d-clasper},  there exists a constant $c_2$ depending only on $M$,  and in particular independent of $K$,  such that for any $K$ we have 
\begin{align}~\label{eqn:c3bound}
|d((S^3, L')*E_K)- d((S^3, H_q)*E_K)| \leq c_2.
\end{align}

\textit{Step 4: Finishing the proof.}
By sequentially using \eqref{eqn:bigbound1},   \eqref{eqn:c2bound},  and \eqref{eqn:c3bound}, we therefore have that for any knot $K$
\begin{align*}
d_{\max}(\Sigma_q(P(K)))& \leq 
d((S^3, L)*E_K)+ c_0+c_1 \\
&\leq d((S^3,  L')*E_K)+c_0+c_1 \\
& \leq d((S^3, H_q)*E_K) +c_0+c_1+c_2.
\end{align*}

Since $H_q= H \cup U_{q-2}$ is a split link,  we have that 
\begin{align*} (S^3, H_q)*E_K& \cong \left[(S^3, H)*E_K\right] \# \left[(S^3, U_{q-2})* E_K \right]\\
&\cong \left[(S^3, H)*E_K\right] \#_{i=1}^{q-2} \left[(S^3, U)*E_K \right]\\
& \cong \Sigma_2(C_{2,1}(K)) \#_{i=1}^{q-2} S^3 \\
&\cong \Sigma_2(C_{2,1}(K)) 
\end{align*}
where the second-to-last equality follows from~\Cref{exl:dbc-c21} and~\Cref{exl:simplestar}. 
So $C_{P,q}:=c_0+c_1+c_2$ is our desired constant. 
\end{proof}

We are now ready to prove~\Cref{thm:mainnullhom},  which states that under the hypotheses of~\Cref{prop:boundingd} we have that $P$ does not induce a homomorphism.

\begin{proof}[Proof of~\Cref{thm:mainnullhom}]
Suppose for a contradiction that $P(T_{2,2k+1})\#P(-T_{2,2k+1})$ is slice for all $k$.  
This implies that the rational homology spheres 
\[Z_k:=\Sigma_q(P(T_{2,2k+1})\#P(-T_{2,2k+1}))= \Sigma_q(P(T_{2,2k+1}))\# \Sigma_q(P(-T_{2,2k+1}))\]
must all bound rational homology balls,  
and hence by \Cref{cor:boundingmeans0} we must have that $d_{\max}(Z_k) \geq 0$ for all $k$. Now, let $C:=C_{P,q}$ be as in \Cref{prop:boundingd},  and observe that 
\begin{align*}
	d_{\scriptscriptstyle\max}(Z_k)&=
	d_{\scriptscriptstyle\max}(\Sigma_q(P(T_{2,2k+1})))+ d_{\scriptscriptstyle\max}(\Sigma_q(P(-T_{2,2k+1}))\\
	& \leq d(\Sigma_2(C_{2,1}(T_{2,2k+1})))+ d(\Sigma_2(C_{2,1}(-T_{2,2k+1})) + 2C\\
	&=d(S^3_{1}(T_{2,2k+1} \# T_{2,2k+1}))+d(S^3_{1}(-T_{2,2k+1} \# -T_{2,2k+1}) + 2C,
\end{align*}
where for the last equality we use the well-known fact that for any knot $K$ one has $\Sigma_2(C_{2,1}(K))\cong S^3_{1}(K \# K^r)$ as was explained in ~\Cref{exl:dbc-c21}, and the fact that torus knots are reversible. 

We can now apply \Cref{thm:surgeryalt} to obtain that
\[	d_{\scriptscriptstyle\max}(Z_k) \leq -2k+2C,\]
which is strictly negative for sufficiently large $k$, providing the desired contradiction.
\end{proof}

\subsection{The case of non null-homologous lifts}~\label{sec:nonnullhomologous}
The arguments of the previous subsection do not generalize well to the setting where the lifts of $\eta$ to $\Sigma_q(P(U))$ represent non-trivial elements of first homology.  We therefore prove~\Cref{thm:maincomp} by working with composite patterns, relying on the elementary observation that if $P \circ R$ does not induce a homomorphism,  then at least one of $P$ and $R$ must also not induce a homomorphism.  

\begin{definition}
Given patterns $P \cup \eta^P$ and $Q \cup \eta^Q$,  the composite pattern $(P \circ Q) \cup \eta^{P \circ Q}$ is defined as  the image of $(P, \eta^Q)$ in  $S^3=E_Q \cup E_{\eta^P}$,
where the identification of $\partial E_Q$ with $\partial E_{\eta^P}$ glues $\lambda_Q$ to $\mu_{\eta^{P}}$ and $\mu_Q$ to $\lambda_{\eta^P}$.  
\end{definition}
See~\Cref{fig:composition} for an example of a composite pattern. 

\begin{figure}[h!]
	\centering
	\def\svgwidth{0.2\textwidth}
	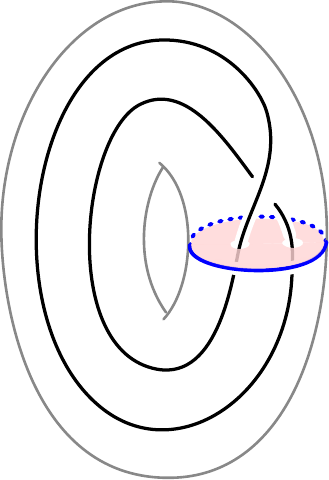\hspace{0.5in}
	\def\svgwidth{0.2\textwidth}
	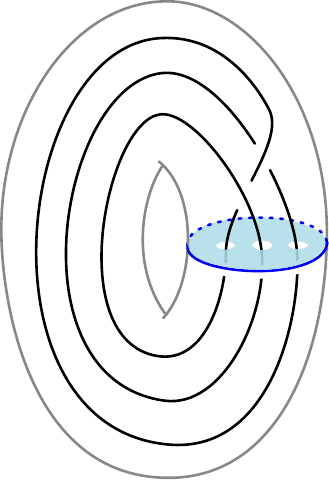\hspace{0.5in}
	\def\svgwidth{0.2\textwidth}
	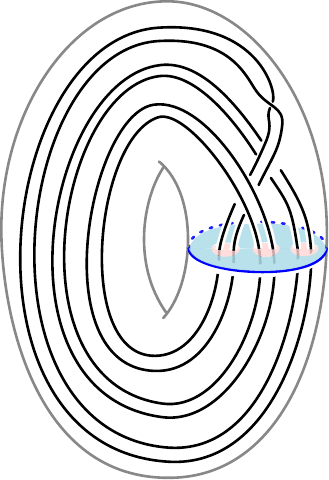
	\caption{Left: the pattern $C\cup \eta^C$ that defines the $(2,1)$-cabling operation, Center: the pattern $A\cup \eta^A$ that defines the $3$-stranded alternating cabling, Right: the composition $\left(C\circ A\right) \cup \eta^{C\circ A}$.}\label{fig:composition}
\end{figure}

\begin{theorem}\label{prop:composition}
Let $P \cup \eta^{\scriptscriptstyle P}$ and $R \cup \eta^{\scriptscriptstyle R}$ be patterns and $n$ divide the winding number $w_R$ of $R$.
Then
\begin{enumerate}
\item $H_1(\Sigma_n((P \circ R)(U))) \cong H_1(\Sigma_n(P(U)))$
\item\label{comp-lambdas} Under this isomorphism, the homology classes $\left[\left(\lambda^{\eta^{\scriptscriptstyle P \circ R}}\right)_i\right]$ correspond to $w_R\left[\left(\lambda^{\eta^{\scriptscriptstyle P}}\right)_i\right]$ for all $i=1, \dots, n$. 
\item\label{comp-lk} $\lk_{\Sigma_n((P\circ R)(U))}(\lambda_i^{\eta^{\scriptscriptstyle P\circ R}}, \lambda_j^{\eta^{\scriptscriptstyle P\circ R}})= w_R^2 \lk_{\Sigma_n(P(U))}(\lambda_i^{\eta^{\scriptscriptstyle P}}, \lambda_j^{\eta^{\scriptscriptstyle P}})$ for all distinct $1 \leq i,  j \leq n$. 
\end{enumerate}
\end{theorem}

We emphasize that in the statement of~\Cref{prop:composition},  we are not assuming that any of the lifts of $\eta_P$, $\eta_R$, or $\eta_{P \circ R}$ are null-homologous, and so the linking numbers are rationally rather than integrally valued. 

\begin{proof} Recall that \Cref{eqn:cyclicofpk} gives decompositions \begin{equation}\label{coverQP}\Sigma_n((P \circ R)(U))=\left(\Sigma_n((P \circ R)(U))\smallsetminus \bigcup_{i=1}^n \nu(\eta^{\scriptscriptstyle P\circ R}_i) \right)
\cup \bigcup_{i=1}^n E_U^{(i)}, \end{equation} and \begin{equation}\label{coverP}\Sigma_n(P(R(U)))=\left(\Sigma_n(P(U))\smallsetminus \bigcup_{i=1}^n \nu(\eta^{\scriptscriptstyle P}_i) \right)
\cup \bigcup_{i=1}^n \left(S^3 \smallsetminus \nu(R\cup \eta^{\scriptscriptstyle R}) \cup E_U\right)^{(i)},\end{equation} 
where the second identification relies on the fact that $n$ divides $w_P$ and the fact that $ \left(S^3 \smallsetminus \nu(R\cup \eta^{\scriptscriptstyle R}) \cup E_U\right)= E_{R(U)}$ as in \Cref{eqn:extofpk}. The first claim follows after observing that $E_U$ and $E_{R(U)}$ have the same homology. Notice that $\lambda^{\eta^{\scriptscriptstyle P\circ R}}_i$ is isotopic to $\lambda^{\eta^{\scriptscriptstyle R}}_i$ in $\Sigma_n((P \circ R)(U))$, since in \Cref{coverQP} $\lambda^{\eta^{\scriptscriptstyle P\circ R}}_i$ is identified with $\mu^{(i)}_{U}$, and in \Cref{coverP} $\mu^{(i)}_{U}$ is identified with $\lambda^{\eta^{\scriptscriptstyle R}}_i$. 

Next, let $D$ be a meridional disk of the solid torus $S^3\setminus \nu(\eta^{\scriptscriptstyle R})$ that intersects $R$ transversely, and for each point $x\in R\cap D$ denote by $N_x$ a disk neighborhood of $x$ in $R\cap D$. The unknotted circles $\partial N_x$ can be divided into two sets $(R\cap D)^{\pm}$ according to whether the orientation agrees with that of $\mu_R$ or not. By the definition of winding number, there is a pairing between the elements of $(R\cap D)^{-}$ and all but $w_R$ elements of $(R\cap D)^{+}$ via tubes contained in $S^3\setminus \nu(\eta^{\scriptscriptstyle R}\cup R)$. If $S$ is the surface obtained as the union of $D\setminus \cup_{x\in R\cap D}N_x$ and large enough tubes, then $S$ is contained in $S^3 \smallsetminus \nu(R\cup \eta^{\scriptscriptstyle R})$ and it's boundary $\partial S$ consists of the disjoint union of $\partial D=\lambda^{\eta^{\scriptscriptstyle R}}$ and $w_R$ copies of negatively oriented meridians of $R$. Since $S^3 \smallsetminus \nu(R\cup \eta^{\scriptscriptstyle R})\subset E_{R(U)}=\left(S^3 \smallsetminus \nu(R\cup \eta^{\scriptscriptstyle R}) \cup E_U\right)$, then for each $i=1,\ldots,n$ there is a copy $S^i$ of $S$ contained in the branched cover $\Sigma_n(P(R(U)))$. It follows that $\lambda^{\eta^{\scriptscriptstyle P\circ R}}_i= \lambda^{\eta^{\scriptscriptstyle R}}_i$ is homologous to $w_R \mu^{(i)}_{R(U)}$ in $E_{R(U)}^{(i)}$ and thus in $\Sigma_n(P(R(U)))$. To establish the second claim it is enough to observe that $\mu^{(i)}_{R(U)}$ is identified with $\lambda^{\eta_{P}}_i$ in $\Sigma_n(P(R(U)))$.

To establish the precise relationship between linking numbers stated as the third claim, notice that for any $j\neq i$ the surface $S^i$ is contained in the complement of $\lambda^{\eta^{\scriptscriptstyle P\circ R}}_j$ and so $$\lk_{\Sigma_n((P\circ R)(U))}(\lambda^{\eta^{\scriptscriptstyle P\circ R}}_i, \lambda^{\eta^{\scriptscriptstyle P\circ R}}_j)= %
w_R \lk_{\Sigma_n((P\circ R)(U))}(\lambda^{\eta_{P}}_i, \lambda^{\eta^{\scriptscriptstyle P\circ R}}_j)=%
w_R^2 \lk_{\Sigma_n((P\circ R)(U))}(\lambda^{\eta_{P}}_i, \lambda^{\eta_{P}}_j).$$
Here we are using the following fact: if $J_1$ and $J_2$ are rationally null-homologous knots which are homologous in the complement of another rationally null-homologous knot $K$ in a 3-manifold $Y$, then $\lk_Y(J_1,K) =\lk_Y(J_2,K) \in \Q$ (and not just in $\Q/\Z$) (see \cite[Chapter 10, Section 77]{st}).

Lastly,  consider an unknotting sequence for $R(U)$: that is, a collection $c_1, \dots, c_M$ of curves in $S^3$ each of which bounds a disk that intersects $R(U)$ in exactly two points and with opposite orientations, and such that $\pm1$ surgery along all these curves transforms $(S^3,  R(U))$ into $(S^3, U)$. Lifting each of these curves to each of the $n$ copies of $E_{R(U)}$ in Equation~\ref{coverP} gives us an $nM$-component link $\Gamma$ in $\Sigma_n((P\circ R)(U))$.  Doing surgery to $\Sigma_n((P\circ R)(U))$ along $\Gamma$ using the lifted framings results in $\Sigma_n(P(U))$. Moreover,  since each component of $\Gamma$ is null-homologous in some $E_{R(U)}^{(i)}$,  this is a null-homologous surgery from $\left(\Sigma_n((P \circ R)(U)), \cup_{i=1}^n {\mu_{R(U)}^{(i)}} \right)$ to $\left(\Sigma_n(P(U)), \cup_{i=1}^n {\mu_{U}^{(i)}} \right)$. Therefore,  since  null-homologous surgeries preserve linking numbers,  we have that $$\lk_{\Sigma_n((P\circ R)(U))}(\mu_{R(U)}^i, {\mu_{R(U)}}^j)= \lk_{\Sigma_n(P(U))}(\mu_{U}^i, \mu_{U}^j) \text{ for all }i \neq j.$$ But $\mu^{(i)}_{R(U)}$ is identified with $\eta^i_P$ in $\Sigma_n(P(R(U))$ and $\mu^{(i)}_U$ is identified with $ \eta^i_P$ in $\Sigma_n(P(U))$, so 
$$\lk_{\Sigma_n((P\circ R)(U))}(\lambda^{\eta_{P}}_i, \lambda^{\eta_{P}}_j)=\lk_{\Sigma_n(P(U))}(\lambda^{\eta_{P}}_i, \lambda^{\eta_{P}}_j)$$
and our desired result follows. \end{proof}

We are now ready to prove \Cref{thm:maincomp}.

\begin{reptheorem}{thm:maincomp}
Let $P$ be a pattern in the solid torus with winding number $w$ and $q$ a prime power dividing $w$. Denote the $q$ lifts of $\eta$ to $\Sigma_q(P(U))$ by $\eta_1,\dots,\eta_q$. 
Suppose that for $i \neq j$, the numbers $\lk(\eta_i, \eta_j)$ are nonnegative and not identically zero,  and let $n$ be the order of $[\eta_1]$ in $H_1(\Sigma_q(P(U))$. 
If $R$ is any pattern whose winding number is a nonzero multiple of $n$, then the composition $P \circ R$ does not induce a pseudo-homomorphism. 
\end{reptheorem}

\begin{proof}
While the lifts of $\eta_P$ may not be null-homologous, if $n$ divides $w_R$ then by \Cref{prop:composition}\eqref{comp-lambdas} the lifts of $\eta_{P\circ R}$ are null-homologous in $\Sigma_q(P\circ R(U))$. Moreover, it follows from \Cref{prop:composition}\eqref{comp-lk} that the lifted curves $\eta_i$ in $\Sigma_q(P \circ R(U))$ also have non-negative linking which is not identically vanishing. \Cref{thm:mainnullhom} now implies that $P \circ R$ cannot be a pseudo-homomorphism. 
\end{proof}

%% file: Composition-cable.pdf_tex
\begingroup%
  \makeatletter%
  \providecommand\color[2][]{%
    \errmessage{(Inkscape) Color is used for the text in Inkscape, but the package 'color.sty' is not loaded}%
    \renewcommand\color[2][]{}%
  }%
  \providecommand\transparent[1]{%
    \errmessage{(Inkscape) Transparency is used (non-zero) for the text in Inkscape, but the package 'transparent.sty' is not loaded}%
    \renewcommand\transparent[1]{}%
  }%
  \providecommand\rotatebox[2]{#2}%
  \newcommand*\fsize{\dimexpr\f@size pt\relax}%
  \newcommand*\lineheight[1]{\fontsize{\fsize}{#1\fsize}\selectfont}%
  \ifx\svgwidth\undefined%
    \setlength{\unitlength}{94.55942345bp}%
    \ifx\svgscale\undefined%
      \relax%
    \else%
      \setlength{\unitlength}{\unitlength * \real{\svgscale}}%
    \fi%
  \else%
    \setlength{\unitlength}{\svgwidth}%
  \fi%
  \global\let\svgwidth\undefined%
  \global\let\svgscale\undefined%
  \makeatother%
  \begin{picture}(1,1.45921229)%
    \lineheight{1}%
    \setlength\tabcolsep{0pt}%
    \put(0,0){\includegraphics[width=\unitlength,page=1]{Composition-cable.pdf}}%
    \put(0.99159061,0.66721701){\color[rgb]{0,0,1}\makebox(0,0)[lt]{\lineheight{1.25}\smash{\begin{tabular}[t]{l}$\eta^{C}$\end{tabular}}}}%
    \put(0.4487973,1.35910274){\color[rgb]{0,0,0}\makebox(0,0)[lt]{\lineheight{1.25}\smash{\begin{tabular}[t]{l}$C$\end{tabular}}}}%
  \end{picture}%
\endgroup%

%% file: Composition-alternating.pdf_tex
\begingroup%
  \makeatletter%
  \providecommand\color[2][]{%
    \errmessage{(Inkscape) Color is used for the text in Inkscape, but the package 'color.sty' is not loaded}%
    \renewcommand\color[2][]{}%
  }%
  \providecommand\transparent[1]{%
    \errmessage{(Inkscape) Transparency is used (non-zero) for the text in Inkscape, but the package 'transparent.sty' is not loaded}%
    \renewcommand\transparent[1]{}%
  }%
  \providecommand\rotatebox[2]{#2}%
  \newcommand*\fsize{\dimexpr\f@size pt\relax}%
  \newcommand*\lineheight[1]{\fontsize{\fsize}{#1\fsize}\selectfont}%
  \ifx\svgwidth\undefined%
    \setlength{\unitlength}{94.65717888bp}%
    \ifx\svgscale\undefined%
      \relax%
    \else%
      \setlength{\unitlength}{\unitlength * \real{\svgscale}}%
    \fi%
  \else%
    \setlength{\unitlength}{\svgwidth}%
  \fi%
  \global\let\svgwidth\undefined%
  \global\let\svgscale\undefined%
  \makeatother%
  \begin{picture}(1,1.45770338)%
    \lineheight{1}%
    \setlength\tabcolsep{0pt}%
    \put(0,0){\includegraphics[width=\unitlength,page=1]{Composition-alternating.pdf}}%
    \put(0.99643767,0.68268188){\color[rgb]{0,0,1}\makebox(0,0)[lt]{\lineheight{1.25}\smash{\begin{tabular}[t]{l}$\eta^{A}$\end{tabular}}}}%
    \put(0.43571761,1.36105297){\color[rgb]{0,0,0}\makebox(0,0)[lt]{\lineheight{1.25}\smash{\begin{tabular}[t]{l}$A$\end{tabular}}}}%
  \end{picture}%
\endgroup%

%% file: Composition.pdf_tex
\begingroup%
  \makeatletter%
  \providecommand\color[2][]{%
    \errmessage{(Inkscape) Color is used for the text in Inkscape, but the package 'color.sty' is not loaded}%
    \renewcommand\color[2][]{}%
  }%
  \providecommand\transparent[1]{%
    \errmessage{(Inkscape) Transparency is used (non-zero) for the text in Inkscape, but the package 'transparent.sty' is not loaded}%
    \renewcommand\transparent[1]{}%
  }%
  \providecommand\rotatebox[2]{#2}%
  \newcommand*\fsize{\dimexpr\f@size pt\relax}%
  \newcommand*\lineheight[1]{\fontsize{\fsize}{#1\fsize}\selectfont}%
  \ifx\svgwidth\undefined%
    \setlength{\unitlength}{94.71481705bp}%
    \ifx\svgscale\undefined%
      \relax%
    \else%
      \setlength{\unitlength}{\unitlength * \real{\svgscale}}%
    \fi%
  \else%
    \setlength{\unitlength}{\svgwidth}%
  \fi%
  \global\let\svgwidth\undefined%
  \global\let\svgscale\undefined%
  \makeatother%
  \begin{picture}(1,1.4568157)%
    \lineheight{1}%
    \setlength\tabcolsep{0pt}%
    \put(0,0){\includegraphics[width=\unitlength,page=1]{Composition.pdf}}%
    \put(0.44014245,1.38326542){\color[rgb]{0,0,0}\makebox(0,0)[lt]{\lineheight{1.25}\smash{\begin{tabular}[t]{l}$C\circ A$\end{tabular}}}}%
    \put(0.98996382,0.67667783){\color[rgb]{0,0,1}\makebox(0,0)[lt]{\lineheight{1.25}\smash{\begin{tabular}[t]{l}$\eta^{C\circ A}$\end{tabular}}}}%
  \end{picture}%
\endgroup%